\documentclass[12pt]{amsart}
\usepackage{amsmath}%
\usepackage[all]{xy}
\usepackage{amssymb}%
\usepackage{amscd,mathtools}%

    \newtheorem{theorem}{Theorem}[section]
    \newtheorem{proposition}[theorem]{Proposition}
    \newtheorem{lemma}[theorem]{Lemma}
    
    \newtheorem*{Conjecture}{Conjecture}
    \newtheorem*{conj3}{Iwasawa Main Conjecture}
    \newtheorem*{Theorem}{Main Theorem}
\theoremstyle{definition}

    \theoremstyle{remark}
    \newtheorem{remark}[theorem]{Remark}
    \numberwithin{equation}{section}
   
\def\Alphabet{A,B,C,D,E,F,G,H,I,J,K,L,M,N,O,P,Q,R,S,T,U,V,W,X,Y,Z}
\def\alphabet{a,b,c,d,e,f,g,h,i,j,k,l,m,n,o,p,q,r,s,t,u,v,w,x,y,z}
\def\endpiece{xxx}
\def\makeAlphabet[#1]{\expandafter\makeA#1,xxx,}
\def\makealphabet[#1]{\expandafter\makea#1,xxx,}
\def\makeA#1,{\def\temp{#1}\ifx\temp\endpiece\else%
\mkbb{#1}\mkfrak{#1}\mkbf{#1}\mkcal{#1}\expandafter\makeA\fi}%
\def\makea#1,{\def\temp{#1}\ifx\temp\endpiece\else\mkfrak{#1}\mkbf{#1}\expandafter\makea\fi}%
\def\mkbb#1{\expandafter\def\csname bb#1\endcsname{\mathbb{#1}}}
\def\mkfrak#1{\expandafter\def\csname fr#1\endcsname{\mathfrak{#1}}}
\def\mkbf#1{\expandafter\def\csname b#1\endcsname{\mathbf{#1}}}
\def\mkcal#1{\expandafter\def\csname c#1\endcsname{\mathcal{#1}}}
\def\makeop[#1]{\xmakeop#1,xxx,}
\def\mkop#1{\expandafter\def\csname #1\endcsname{{\mathrm{#1}}}} %
\def\xmakeop#1,{\def\temp{#1}\ifx\temp\endpiece\else\mkop{#1}\expandafter\xmakeop\fi}%
\makeAlphabet[\Alphabet]
\makealphabet[\alphabet]
\makeop[Alt,End,Ext,Hom,Sym,Tor,Aut]
\makeop[CH,Pic,div,Det]
\makeop[Ker,Im,Coker,Coim,id,pr,isom]
\makeop[Spec,Spf,Spm]
\makeop[Re,Im]%
\makeop[dR,Nis,crys,rig,syn,tor,Cont]
\makeop[Gal,Res,ab]
\makeop[pol]
\makeop[Var,an,Isoc,univ]%
\makeop[Gl,Sl]
\makeop[Lie]
\makeop[tay,Tr,rk,ind]
\makeop[Sat]
\makeop[Frob,Art]
\makeop[Det]

\begin{document}


\title[ENTC for almost abelian extenstions]{The local equivariant Tamagawa number conjecture for almost abelian extensions}
\author[Johnson-Leung]{Jennifer Johnson-Leung}
\address{Department of Mathematics, University of Idaho, Moscow, Idaho}
\email{jenfns@uidaho.edu}
\thanks{The author was supported in part by an NSA Young Investigator's Award}


\date{\today}

\begin{abstract}
 We prove the local equivariant Tamagawa number conjecture for the motive of an abelian extension of an imaginary quadratic field with the action of the Galois group ring for all split primes $p\neq 2, 3$ at all integer values $s<0$.
\end{abstract}
\maketitle

\section{Introduction}
Since Dirichlet's remarkable proof of the analytic class number formula in the first half of the nineteenth century,  conjectures on the relationship between the values of $L$-functions and invariants of arithmetic objects have motivated a great deal of research.  The equivariant Tamagawa number conjecture (ETNC)  is a unifying statement concerning the special values of motivic $L$-functions which encompasses both the Birch and Swinnerton-Dyer conjecture and the generalized Stark conjectures.  It is a deep and sweeping assertion which has yielded to proof in very few cases.  The Tamagawa number conjecture builds on the conjectures of Beilinson \cite{Bei},  predicting that the $L$-values of smooth projective varieties over $\bbQ$ are given by period integrals and regulator maps, up to a rational factor, $q$.   Bloch and Kato \cite{BK}, further predicted that the rational number $q$ is given in terms of Tamagawa numbers and the order of a certain Tate-Shafarevic group.  This conjecture was reformulated by 
Fontaine and Perrin-Riou \cite{FPR} in a language that was naturally extended to motives with extra symmetries by Burns and Flach \cite{BF2, BF3}.  There are two equivalent formulations of the conjecture.  The first is a global formulation that concerns of the vanishing of a certain element in relative $K$-theory.  The second is a local formulation that concerns the equality of two lattices.
\par
In this paper, we study the conjecture for the motive of an abelian extension of an imaginary quadratic field.  We call these almost abelian extensions because there are many similarities to the case of absolutely abelian extensions stemming from the fact that an imaginary quadratic field has only one archimedean place.  Notice that as this place is complex, the local conjecture at the prime 2 will be less complicated than in the case of absolutely abelian extensions \cite{F4}.  However, we do not consider the prime 2 in this paper.  Our main result is a proof of the local ETNC at all split primes $\ell\nmid 6$ at negative integer values of the $L$-function.  Bley also considers the case of abelian extensions of imaginary quadratic fields for the $L$-value at 0  \cite{Bl}.  His proof has restrictions similar to those in this work stemming from the vanishing of the $\mu$-invariant of a certain Iwasawa module.  It would be quite nice to prove compatibility of the conjecture with the functional equation for 
this class of motives as well, as the combination of these results would give the conjecture at any integer value of the $L$-function.
\par
The only completely proven case of the equivariant Tamagawa number conjecture is the proof of Burns, Flach, and Greither for abelian extensions of $\bbQ$ \cite{BG, BF4, F4}.  Huber and Kings proved independently a slightly weaker version of this cyclotomic case \cite{HK2} which has since been strengthened to a full proof by work of Witte \cite{Wi}.  
Even partial results are quite more sparse.  Burns and Flach give a proof for an infinite family of quaternion extensions \cite{BF3}, and Navilarekallu gives a method of proof for $A_4$ extensions which he employs for a specific case \cite{Na}. There are also several theorems that are not equivariant.  Gealy recently proved a weakened version of the Tamagawa number conjecture for modular forms of  weight greater than 1 \cite{Ge}.  Kings also proved a weakened version for elliptic curves with CM by an imaginary quadratic field of class number 1 \cite{K2}.   In both of these cases, the conjecture must be weakened because it is not known that the motivic cohomology groups are finitely generated.  By working with the constructible part of the group, however, a proof can be given.  Bars builds on work of Kings to give some non-equivariant results for Hecke characters of imaginary quadratic fields \cite{Ba}.  The survey papers of Flach \cite{F2, F3} include a nice formulation of the local version of the 
equivariant Tamagawa number conjecture for arbitrary motives over $\bbQ$ and discusses the proven cases.  We strive here to keep notation consistent with this overview.  
\par
This paper is an improvement of the main result in the author's thesis, and so many thanks are due her thesis advisor, Matthias Flach.  She would also like to thank Werner Bley, Matthew Gealy, and Guido Kings for very helpful conversations and the referee for a careful reading of the manuscript.
\subsection{Notation}  
Let $K$ be an imaginary quadratic field with ring of integers $\cO_K$ and let $\frf$ be an integral ideal of $\cO_K$.  We will let $K(\frf)$ denote the ray class field of $K$ of conductor $\frf$. By a $CM$ pair of modulus $\frf$ over a number field $F$, we mean a pair $(E,\alpha)$ where $E$ is an elliptic curve over $F$ with complex multiplication by $\cO_K$ and such that the inculsion of $\cO_K$ into $F$ factors through $\End(E)$, and $\alpha\in E(\bbC)$ is a primitive $\frf$-division point.  By \cite[15.3.1]{Ka}, there is a CM pair of modulus $\frf$ over $K(\frf)$ which is isomorphic to $(\bbC/\frf,1\bmod\frf)$ over $\bbC$.  This pair is unique up to isomorphism, and whenever $\cO_K^\times\rightarrow(\cO_K/\frf)^\times$ is injective the isomorphism is unique.  Denote $(\bbC/\frf, 1\bmod \frf)$ the canonical CM pair. 
\par
We will make repeated use of the graded determinant functor $\Det$ of Knudsen and Mumford \cite{KM}.  Let $R$ be a commutative ring, and $P$ a projective $R$-module.
The determinant of $P$ is the invertible $R$-module
$$\Det_RP:=\bigwedge_{R}^{\mbox{rk}_{R}{P}}P.$$
If $C: \cdots\rightarrow P^{i-1}\rightarrow P^i\rightarrow P^{i+1}\rightarrow\cdot\cdot\cdot$ is a perfect complex of projective $R$-modules,
the determinant of the complex is defined to be the graded invertible $R$-module 
$$
\Det_RC:=\bigotimes_{i\in\bbZ}\Det_R^{(-1)^i}P^i
$$
and depends only on the quasi-isomorphism class of $C$.  Indeed, if the cohomology groups $H^i(C)$ are all perfect, one has
$$
\Det_RC=\bigotimes_{i\in\bbZ}\Det_R^{(-1)^i}H^i(C).
$$
\section{The Main Theorem}
\noindent Let $F$ be an abelian extension of $K$ with Galois group $G$. 
We consider the Chow motive $M=h_0(\Spec(F))(j)$ where $j$ is a negative integer. 
\subsection{The local statement of the ETNC} We will formulate the equivariant Tamagawa number conjecture (ETNC) for this motive.  $M$ carries an action of the semisimple $\bbQ$-algebra $A=\bbQ[G]$.  We study $M$ via its realizations and the action of $A$ on these spaces, focusing on the Betti realization 
$$M_B=H^0(\Spec(F)(\bbC), \bbQ(j)),$$
 which carries an action of complex conjugation, the de Rham realization 
$$M_{dR}=H^0_{dR}(\Spec(F)/\bbQ)(j),$$
 with its Hodge filtration, and the $\ell$-adic realization 
$$M_\ell=H^0_{et}(\Spec(F)\times_{\bbQ}\bar{\bbQ},\bbQ_\ell(j))),$$
 which is a continuous representation of $\Gal(\bar{\bbQ}/\bbQ)$.
The $A$-equivariant $L$-function of $M$ is defined via an Euler product
$$
L(_{A}M,s)=\prod_p\Det_A(1-\Frob_p^{-1}\cdot p^{-s}\mid M_\ell^{I_p})^{-1}.
$$
The leading term of the Taylor expansion at $s=0$ decomposes over the characters of $G$
$$
L^*(_{A}M)=(L'(\eta,j))_{\eta\in\hat{G}}\in (A\otimes_\bbQ\bbR)^\times.
$$
 We can now introduce one of the key objects in this formulation of the Tamagawa number conjecture: the {\em fundamental line} is the $A$-module
$$
\Xi(_AM)=\Det_A(K_{1-2j}(\cO_F)^*\otimes_\bbZ\bbQ)\otimes_A\Det_A^{-1}(M_B^+),
$$
where $+$ denotes the invariants under complex conjugation and $K_{1-2j}(\cO_F)^*$ is the dual of the algebraic $K$-group $K_{1-2j}(\cO_F)=K_{1-2j}(F)$.  This ``line'' is the tool which enables the comparison of the $L$-value with algebraic invariants of the number field.
\par
Borel's regulator \cite{Bor}, is an isomorphism
\[ K_{1-2j}(\cO_{F})\otimes_{\bbZ}\bbR\xrightarrow{\rho_\infty}
\left(\bigoplus_{\sigma\in\cT}\bbC/\bbR\cdot(2\pi i)^{1-j}\cdot \sigma\right)^+\]
where $\cT=\Hom(F,\bbC)$.  Since $j<0$, $K_{1-2j}(\cO_{F})\simeq K_{1-2j}(F)$. For an element, $\sum_{\sigma\in\cT}x_\sigma\cdot\sigma$, the Galois group acts via
$
g\cdot\left(\sum_{\sigma\in\cT}x\cdot \sigma\right)=\sum_{\sigma\in\cT}x\cdot g^{-1}\sigma.
$
With this action, $\rho_\infty$ is $A$-equivariant just as in the case of the Dirichlet regulator.  Now,the $\bbR$-dual of $\left(\bigoplus_{\sigma\in\cT}\bbC/\bbR\cdot(2\pi i)^{1-j}\cdot \sigma\right)^+$ is identified with $M_{B}^+\otimes_\bbQ\bbR$ by taking invariants in the $\Gal(\bbC/\bbR)$-equivariant perfect pairing
\[ \bigoplus_{\sigma\in\cT}\bbR\cdot(2\pi
i)^j \times \bigoplus_{\sigma\in\cT}\bbC/\bbR \cdot(2\pi i)^{1-j}\to
\bigoplus_{\sigma\in\cT}\bbC/ 2\pi
i\cdot\bbR\xrightarrow{\Sigma}\bbR\] induced by multiplication.  
Hence, the dual of the Borel regulator induces an $A$-equivariant isomorphism 
$$
\vartheta_\infty:A\otimes_\bbQ\bbR\rightarrow\Xi(_AM)\otimes_\bbQ\bbR.
$$
Note that the $L^*(_AM)$ lies in the domain of this isomorphism, and Gross conjectured that its image lies in the rational space $\Xi(_AM)\otimes_\bbQ1$ \cite{N}.  This ``Stark-type'' conjecture says that up to a rational factor, the $L$-value is given by the Borel regulator and was proved by Deninger \cite{D2} in his work on the Beilinson conjectures for Hecke characters of imaginary quadratic fields.  
\par
Fix a prime number $\ell$, let $S$ be a set of primes containing $\ell$, $\infty$, and the primes which ramify in $F$, and let $A_\ell=A\otimes_\bbQ\bbQ_\ell$.  We now concern ourselves with the $\ell$-part of the rational factor by considering the isomorphism induced by the Chern class map and the cycle class map
$$
\vartheta_\ell:\Xi(_AM)\otimes_\bbQ A_\ell\rightarrow\Det_{A_\ell}(R\Gamma_c(\bbZ[\frac{1}{S}],M_\ell)),
$$
where the right hand side denotes the cohomology with compact supports as defined by the mapping cone
$$
R\Gamma_c(\bbZ[\frac{1}{S}],M_\ell)\rightarrow R\Gamma(\bbZ[\frac{1}{S}],M_\ell)\rightarrow\bigoplus_{p\in S}R\Gamma(\bbQ_p,M_\ell),
$$
The conjecture then compares a natural lattice in the right hand side of this isomorphism to the lattice generated by the image of $L^*(_AM)$.  To construct the lattice, we choose the order $\bbZ[G]$ in $A$ and the $\Gal(\bar{\bbQ}/\bbQ)$-stable projective $\bbZ_\ell[G]$-lattice 
$$
T_\ell=H^0(\Spec(F\otimes_\bbQ\bar{\bbQ}),\bbZ_\ell(j)).
$$
\begin{Conjecture}[Local ETNC]
 There is an equality of lattices 
$$
\vartheta_\ell\vartheta_\infty(L^*(_AM)^{-1})\cdot\bbZ_\ell[G]=\Det_{\bbZ_\ell[G]}R\Gamma_c(\bbZ[\frac{1}{S}],T_\ell).
$$
inside of $\Det_{A_\ell}R\Gamma_c(\bbZ[\frac{1}{S}], M_\ell)$.
\end{Conjecture}
\noindent The ETNC for number fields is equivalent to the statement that the local conjecture holds at every prime number $\ell$.  This determines $L^*(_AM)$ up to a unit in $\bbZ[G]$. Notice that the ETNC depends on the choice of order but is independent of the choice of $S$ and $T_\ell$ \cite{F1}.  This indepence of lattice is exploited to prove the main results of \cite{JLK} which will be an important ingredient in the proof of our main theorem.  

\begin{Theorem} Let $F$ be an abelian extension of an imaginary quadratic field $K$ with Galois group $G$.  Then the local equivariant Tamagawa number conjecture is valid for the motive $h^0(\Spec(F))(j)$ for $j<0$ at every rational prime $p\nmid 6$ which splits in $K$. 
\end{Theorem}
\begin{remark} The restriction to split primes can be lifted whenever the $\mu$-invariant of a certain Iwasawa module can be shown to vanish, as discussed in Section \ref{Iwasawa}.
\end{remark}

\subsection{Proof strategy} We first reduce to the case that $F=K(\frm)$ is the ray class field of conductor $\frm$ where the only root of unity in $K$ which is congruent to 1 modulo $\frm$ is 1 by applying the general functoriality result of Burns and Flach \cite[Prop 4.1 b)]{BF2}. Let $G_\frm$ denote the Galois group $\Gal(K(\frm)/K)$.    
The conjecture asserts an equality of rank 1 $\bbZ_\ell[G_\frm]$-modules inside of $\Det_{A_\ell}R\Gamma_c(\bbZ[\frac{1}{S}], M_\ell)$.  Our strategy is to compute a generator of each of these modules.  By a rational character of $G_\frm$, we mean an $\Aut(\bbC)$ orbit of complex characters of $G_\frm$.  The group ring $A=\bbQ[G_\frm]$, splits as a product of number fields indexed by these rational characters $\chi$ of $G_\frm$, and thus $A_\ell$ splits as well.  Hence, it suffices to compare our generators character by character.  In section \ref{elements} we compute the image of $L^*(_AM)$ under the composition $\vartheta_\ell\vartheta_\infty$.  In section \ref{Iwasawa} we compute a basis of $\Det_{\bbZ_\ell[G]}R\Gamma_c(\bbZ[\frac{1}{S},T_\ell)$ 
via descent from the Iwasawa main conjecture and show that it coincides with the image of $L^*(_AM)$, completing the proof.
\section{The image of the $L$-value}\label{elements}
Let $\chi$ be a character of $G_\frm$ of conductor $\frf_\chi$.  If $\frf_\chi\neq\frm$, then $\chi$ is induced from a character of $G_{\frf_\chi}$.  The Dirichlet $L$-function of $\chi$ differs from the Artin $L$-function of $\chi$ as a $G_\frm$ representation by a finite number of Euler factors
$$
L_{D}(\chi,s)=\left(\prod_{p\mid\frm, p\nmid f_\chi}(1-\chi(\frp)\cN\frp^{-s})\right)L(\chi,s).
$$
 In \cite{D1,D2} Deninger constructs elements in motivic cohomology in order to prove the Beilinson conjecture for Hecke characters of an imaginary quadratic field.  We will use these motivic elements to prove our main theorem, but as we seek a finer result about the value $L^*(_AM)$, we will have to revisit some details of the proof as well.  We first translate our question to the setting of Hecke characters by twisting $\chi$ by the norm character of $K$ to obtain a Hecke character of weight 2,
$$
\varphi_\chi=\chi N_{K/\bbQ}.
$$
Let $E$ be an elliptic curve defined over $K(\frm)$ with complex multiplication by $\cO_K$ where $E$ has the additional property that the Serre-Tate character factors through the norm map from $K(\frm)$ to $K$.  Let $\cA=R_{K(\frm)/K}E$ be the Weil restriction of the elliptic curve.  Then $\cA$ is an abelian variety over $K$ with CM by a semisimple $K$-algebra $T$ and  with Serre-Tate character $\varphi_{\cA}$.  Deninger proves that any Hecke character $\varphi$ of weight $w>0$ is of the form 
$$
\prod_{i=1}^w\lambda_i\circ\varphi_{\cA}=\prod_{i=1}^w\varphi_{\lambda_i}
$$
where $\lambda_i\in\Hom(T,\bbC)$, \cite[Prop 1.3.1]{D2}.  We choose once and for all a type $(1,0)$ character $\varphi$ with $N_{K/\bbQ}=\varphi\bar{\varphi}$ and take $\frm$ to be a multiple of the conductor of $\varphi$.  Then we have
$$
\varphi_\chi=\varphi_{\lambda_1}\varphi_{\lambda_1}
$$
and $\frm$ is a multiple of the conductors of $\lambda_1$ and $\lambda_2$.

\subsection{Torsion points} Deninger computes the special values of the $L$-function of such Hecke characters in terms of torsion points on $E$.  Let $\frf$ be an ideal of $\cO_K$ and let $\rho_\frf\in\bbA^*_K$ be an id\`ele with ideal $\frf$.  Choose an approximation $f_\frf\in K^*$ with
\begin{equation}\label{approximation}
\begin{array}{lr}
v_\frp(f_\frf)\leq0& \text{if} \, \frp\nmid\frf\\
v_{\frp}(f^{-1}_\frf-(\rho_\frf)^{-1}_\frp)\geq0& \text{if} \, \frp\mid\frf.
\end{array}
\end{equation}
We also fix an isomorphism 
$$
\theta_E:\cO_K\simeq\End_{K(\frm)}(E)
$$
such that $\theta^*_E(k)\omega=k\omega$ for all $\omega\in H^0(E,\Omega^1_{E/K(\frm)})$ and an embedding $\tau_0$ of $K(\frm)$ into $\bbC$ such that $j(E)=j(\cO_K)$.  Then we have a complex isomorphism 
$$
E(\bbC)\simeq\bbC/\Gamma\quad \text{where} \quad \Gamma=\Omega\cO_K
$$
for some $\Omega\in\bbC$.  This choice is non-canonical and determines a class in the Betti cohomology of $E$.  If $z\in\bbC$, we let $([z])$ denote the point on $E$ under this isomorphism.  Now ${}_\frf\beta=([\Omega f_\frf^{-1}])$ is a point in $E[\frf]$ which is rational over $K(\frf)$.
\par
Fix a set of ideals 
$$
\{\frb_g\subseteq\cO_K\}_{g\in G_\frm}
$$
with Artin symbol $(\frb_g,K(\frm)/K)=g\in G_\frm$.
For $g\in G_\frm$, let ${}^gE$ be the curve obtained by base change according to the diagram
$$
\xymatrix{
{^g}E \ar[r]\ar[d]& E\ar[d] \\
\Spec(K(\frm))\ar[r]^{g^*} & \Spec(K(\frm)) 
}
$$
We denote the period lattice of ${}^gE$ by $\Gamma_g$.
Following Deninger, for any ideal $\fra\subseteq\cO_K$ which is prime to the conductors of $\varphi_{\lambda_1}$ and $\varphi_{\lambda_2}$ we define $\Lambda(\fra)\in K(\frm)^\times$ by
$$
\varphi_\cA(\fra)^*\omega^{(\fra,K(\frm)/K)}=\Lambda(\fra)\omega
$$
where $\omega^g\in H^0(E^g,\Omega^1)$ has period lattice $\Gamma_g$ and $\varphi_\cA(\fra)\in T^\times$ is viewed as an isogeny $E\rightarrow {}^{(\fra,K(\frm)/K)}E$.  Now we can consider a family of $\frf$-torsion points on the conjugates of $E$,
$$
{}_\frf\beta_g=([\Lambda(\frb_g\Omega f^{-1}_\frf)])
$$
with an action of $G_\frm$ given by ${^h}{_\frf}\beta_g={_\frf}\beta_{hg}$.
\begin{proposition}\cite[(3.4)]{D2}\label{denformula}
Let $\chi$ be a character of $\Gal(K(\frm)/K)$ of conductor $\frf_\chi=\frf$.  The Artin $L$-series $L(\chi,s)$ has a first order zero for every $s=j<0$, and the special value is given by the formula
\begin{multline*}
L'(\chi,j)=(-1)^{j}\frac{\Phi(\frf)(-j)!^2}{\Phi(\frm)}\left(\frac{\sqrt{d_K}\cN\frf}{2\pi i}\right)^{-j}\chi(\rho_\frf)\\\sum_{g\in G_\frm}\chi(g)A(\Gamma_g)^{1-j}\sum\limits_{0\neq\gamma\in\Gamma_g}\frac{({}_\frf\beta_g,\gamma)_g}{|\gamma|^{2-2j}}, 
\end{multline*}
where  $\Phi$ is the totient function, $d_K$ is the discriminant of $K$, for any $\bbZ$-basis of $\Gamma_g$ with Im($u/v)>0$, $A(\Gamma_g)=(\bar{u}v-\bar{v}u)/2\pi i)$, and 
$$
(,)_g:\bbC/\Gamma_g\times\Gamma_g\rightarrow U(1)
$$ 
given by $(z,\gamma)_g=exp(A(\Gamma_g)^{-1}(z\bar{\gamma}-\bar{z}\gamma))$ is the Pontrjagin pairing.
\end{proposition}
This formula is the essential first step to proving the ETNC as it describes the $L$-value in terms of points on an elliptic curve.  These points can then be used to construct elements in motivic cohomology.  Notice that the formula is unchanged if we consider $\chi$ to be a rational character of $G_\frm$ in the sense that it represents equality of tuples
\begin{multline*}
(L'(\eta,j))_{\eta\in\chi}=\big((-1)^{j}\frac{\Phi(\frf)(-j)!^2}{\Phi(\frm)}\left(\frac{\sqrt{d_K}\cN\frf}{2\pi i}\right)^{-j}\eta(\rho_\frf)\\\sum_{g\in G_\frm}\eta(g)A(\Gamma_g)^{1-j}\sum\limits_{0\neq\gamma\in\Gamma_g}\frac{({}_\frf\beta_g,\gamma)_g}{|\gamma|^{2-2j}}\big)_{\eta\in\chi}, 
\end{multline*}
Where $\eta$ is a complex character which is in the orbit represented by the rational character $\chi$.

\subsection{Eisenstein Symbol} The Eisenstein symbol, originally constructed by Beilinson, is roughly a map from torsion points of an elliptic curve to the cohomology of a power of the curve.  There are several variations of Eisenstein symbol in the literature.  In particular, Deninger uses a variation for which the domain is divisors of degree zero and defines a degree zero divisor 
$$
{}_\frf\beta'={}_\frf\beta+\frac{1}{\tilde{N}^{4-2j}-1}(0)-\frac{\tilde{N}^{2-2j}}{\tilde{N}^{4-2j}-1}\sum_{p\in E(\bbC)[\tilde{N}]}(p)
 $$
to construct the motivic elements.  Here $\tilde{N}\geq2$ is an auxiliary integer.  However, this approach would cause a technical difficulty for us when we consider the $\ell$-adic regulator.  Thus, we introduce our variation in the following lemma and show that it is compatible with Deninger's construction.  Note that this is used implicitly in work of Kings \cite{K2}, which we will discuss in section \ref{ladic}.

\begin{lemma}\label{eisenstein}
Let $E$ be an elliptic curve.  For any $k>0$, there is a variation of the Eisenstein symbol $\cE_\cM^k:\bbQ[E[\frf]\setminus0]\rightarrow H^{k+1}_\cM(E^k,k+1)$ which is defined for divisors of any degree.  Moreover, 
$$\cE_\cM^k({_\frf}\beta^\prime)=\cE_\cM^k({_\frf}\beta).$$
\end{lemma}
\begin{proof}
For $N=\cN\frf\geq3$, let $M$ be the modular curve parameterizing elliptic curves with full level $N$ structure, and let $\mathfrak{E}$ be the universal elliptic curve over $M$.  Choose a level $N$ structure on $E$, $\alpha:(\bbZ/N\bbZ)^2\stackrel{\sim}{\rightarrow}E[N]$, rational over some extension $K^\prime$ of $K(\frm)$.  By the universality of $\mathfrak{E}$, we have the following diagram depending on the choice of level-$N$ structure.
$$
\xymatrix {
E\ar[r]^{\alpha^*}\ar[d]&\mathfrak{E}\ar[d]\\
\Spec(K^\prime)\ar[r]&M
}
$$
Denote by $\tilde{\mathfrak{E}}^0$ the fiber over the cusps of the connected component of the generalized elliptic curve over the compactification of $M$.  Then we can define ${\rm Isom}={\rm Isom}(\bbG_m,\tilde{\mathfrak{E}}^0_{\rm Cusp})$.  ${\rm Isom}$ is a $\mu_2$ torsor over the subscheme of cusps, and we consider the subset $\bbQ[{\rm Isom}]^{(k)}\subseteq \bbQ[{\rm Isom}]$ where $\mu_2$ acts by $(-1)^k$.  Define the horospherical map  $\varrho^k:\bbQ[\mathfrak{E}[N]]^0\rightarrow\bbQ[{\rm Isom}]^{(k)}$ explicitly by
$$
\varrho^k(\psi)(g)=\frac{N^k}{k!(k+2)}\sum_t\psi(g^{-1}t)B_{k+2}(\frac{t_2}{N}),
$$
where $t=(t_1,t_2)\in(\bbZ/N\bbZ)^2$ and $B_{k}(x)$ is the $k$th Bernoulli polynomial.  When $k>0$, $\varrho$ is well-defined for divisors of any degree.
\par
For an elliptic curve over any base, Beilinson \cite{Bei} constructs an  Eisenstein symbol $\cE_\cM^k:\bbQ[E[N]]^0\rightarrow H^{k+1}_\cM(E^k,k+1)$ which is preserved under base change.  For the universal elliptic curve $\mathfrak{E}$, we also have a boundary map 
$$
\rm{res}^k:H^{k+1}_\cM(\mathfrak{E}^k,k+1)\rightarrow\bbQ[{\rm Isom}]^{(k)}
$$
 coming from the long exact cohomology sequence, and another map also called the Eisenstein symbol 
$$
\rm{Eis}^k:\bbQ[{\rm Isom}]^{(k)}\rightarrow H^{k+1}_\cM(\mathfrak{E}^k,k+1)
$$
with $\rm{res}^k\circ\rm{Eis}^k=id$. The following diagram commutes when restricting to degree zero divisors.
$$
\xymatrix{\bbQ[\mathfrak{E}[N]]\ar[r]^{\varrho}&\bbQ[{\rm Isom}]^{(k)}\ar[r]^{\rm Eis}& H^{k+1}_\cM(\mathfrak{E}^k,k+1)\ar[d]_{\alpha^*}\\
\bbQ[E[N]]^0\ar[rr]^{\cE_\cM^k}\ar[u]^{\alpha}&& H^{k+1}_\cM(E^k,k+1)
}
$$
Indeed, the horospherical map above was computed by Schappacher and Scholl to be the composition $\cE_\cM^k\circ\rm{res}^k$ \cite{SS}.  Combining this fact with base change, the diagram commutes, and we can compute the Eisenstein symbol at torsion points on the elliptic curve. 
 Moreover, this computation does not depend on the choice of full level structure since the assignment of Eisenstein symbols commutes with the $GL_2$ action on the torsion sections and is thus invariant under the trace $Y(N)\rightarrow Y_1(N)$ \cite[Lemma 3.1.2]{K1} 
 \par
 To show that $\cE_\cM^k({_\frf}\beta^\prime)=\cE_\cM^k({_\frf}\beta)$, it suffices to show that 
 $$
 {_\frf}\beta^\prime-{_\frf}\beta\in\ker\varrho,
 $$
As the action of $\mathbb{G}_m$ preserves the identity section on the curve,
 $$
 \varrho^{-2j}(0)(g)=\frac{N^{-2j}}{(-2j)!(2-2j)}B_{2-2j}(0).
 $$
We compute
 \begin{align*}
 \varrho^{-2j}\left(\sum_{p\in E(\bbC)[a]}(p)\right)(g)&=\frac{a^{-2j}}{(-2j)!(2-2j)}\sum_{(p)=(t_1,t_2)\in(\bbZ/a\bbZ)^2}B_{2-2j}\left(\frac{t_2}{a}\right)\\
 &=\frac{a^{1-2j}}{(-2j)!(2-2j)}\sum_{i=0}^{a-1}B_{2-2j}\left(\frac{i}{a}\right).
 \end{align*}
  Moreover, the distribution relation
 $$
 B_k(X)=a^{k-1}\sum_{i=0}^{a-1}B_k\left(\frac{X+i}{a}\right)
 $$
 implies that
$$
\varrho^{-2j}\left(\sum_{p\in E(\bbC)[a]}(p)\right)(g)=\frac{1}{a^{2-2j}}\varrho^{-2j}(0)(g),
$$
which completes the proof of the lemma.
 \end{proof}

In order to study Hecke characters of $K(\frm)$, we must consider the image of the elements above in the cohomology of the number field.  To this end, Deninger constructs
the Kronecker map, $\mathcal{K}_\cM$, which is a projector given by the composition
$$
\xymatrix {
H_{\mathcal{M}}^{1-2j}(E^{-2j},1-2j)\ar[r]^{(id,\theta_E(\sqrt{d_K})^{-j,*}}\ar[dr]_{\mathcal{K}_\mathcal{M}}&H^{1-2j}_{\mathcal{M}}(E^{-j},1-2j)\ar[d]^{\pi_{-j,*}}\\
&H^1_{\mathcal{M}}(\Spec(K(\frm)),1-j),
}
$$
where the map $\pi_{-j*}$ is a proper push forward.

\subsection{Deninger's Theorem}
We now have the tools to prove the following adaptation of \cite[Theorem 3.1]{D2}.
\begin{theorem}\label{motivic}
For every ideal $\frf\mid\frm$, there are motivic elements 
$$\xi_\frf(j)\in H^1_\cM(K(\frm), 1-j)$$
with the property that if $\chi$ is a rational character of $G_\frm$ of conductor $\frf$, then
$$e_\chi(\rho_\infty(\xi_\frf(j)))=\frac{(2\cN\frf)^{-(1+j)}\Phi(\frm)}{(-1)^{1+j}(-2j)!\Phi(\frf)}L^\prime(\bar{\chi},j)\eta_\bbQ.$$
where $\eta_\bbQ=e_\chi\cdot\tau_0$ is a basis of the $\chi$-component of $\left(H^0_B(\Spec(F)(\bbC),\bbQ(j))^+\right)^*$, and 
$\Phi$ is Euler's totient function.  Moreover, the resulting elements the Betti cohomology form a norm-compatible system in the following sense:  If $\chi$ is a character of conductor $\frf$, then 
$$
e_\chi(\rho_\infty(w_\frf/w_{\frp\frf}\Tr_{K(\frp\frf)/K(\frf)}\xi_{\frf\frp}(j))=\begin{cases}
                                                                                      e_\chi(\rho_\infty(\xi_{\frf}(j))&\frp\mid\frf\\
(1-\bar{\chi}(\frp)\cN\frp^{-j})e_\chi(\rho_\infty(\xi_{\frf}(j))&\frp\nmid\frf.
\end{cases}
$$
\end{theorem}
\begin{remark}The trace map in the above theorem should be understood as corestriction.  Moreover, this compatibility should already hold for the elements $\xi_{\frf}(j)$, as under the $\ell$-adic regulator these are the pullback of the polylogarithm sheaf.  In fact, Scholl proves some compatibility for the Eisenstein symbol on the universal elliptic curve in \cite[A.2]{Sch}, but the above theorem is sufficient for our purposes.
\end{remark}
\begin{proof}
Recall that we have fixed a choice of an embedding $\tau_0:K\hookrightarrow\bbC$ and of a uniformization $E\simeq\bbC/\Omega\cO_K$.  The  torsion point ${_\frf}\beta$ is dependent on the choice of id\`ele $\rho_\frf$.  By the main theorem of complex multiplication, the Artin symbol
$$
\Art(\rho_\frf)^{-1}:E\rightarrow{^{\Art(\rho_\frf)}}E
$$
maps the pair $(E(\bbC),{_\frf}\beta)$ to $(\bbC/\Omega\frf,1\bmod\Omega\frf)$ since ${_\frf}\beta=\Omega f_\frf^{-1}$ and the ideal
$$
f_\frf^{-1}(\rho_\frf) \equiv1\bmod\frf.
$$
Indeed, the restrictions on the valuation of $f_\frf$ at each prime $\frp\mid\frf$ in (\ref{approximation}), give that 
$$
\rho_{\frf,\frp}/f_\frf\in1+\frm_\frp^{\mathrm{ord}_\frp\frf}
$$
where $\frm_\frp$ is the maximal ideal in the local ring $\cO_{K_\frp}$.
Moreover, one may choose the id\`eles $\rho$ to be multiplicative in the sense that $\rho_{\frf\frp}=\rho_\frf\rho_\frp$.
\par
Thus we define our motivic elements
$$
\xi_\frf(j):=\begin{cases}\mathcal{K}_\cM\cE^{-2j}(\Art(\rho_\frf)^{-1}{_\frf}\beta)& \cN\frf\geq3\\\mathcal{K}_\cM\cE^{-2j}(\Art(\rho_\frf)^{-1}{_\frf}\beta')& \cN\frf\leq2\end{cases}
$$
where ${_\frf}\beta'$ is the degree zero divisor used by Deninger.  Notice that the computation of the $\ell$-adic regulator is not possible for these divisors, but the norm compatibility allows us to bypass this difficulty.
\par
For an $\frf$-torsion point of $E$, we define the function
$$
\mathcal{M}_j(x)=\sum_{0\neq\gamma\in\Gamma_g}\frac{(x,\gamma)_g}{|\gamma|^{2(1-j)}}.
$$
 In \cite[3.2]{D2}, Deninger computes that for an embedding $\tau$ of $F$ into $\bbC$,
$$
\rho_\infty(\mathcal{K}_\cM \cE_\cM^{-2j}({_\frf}\beta^\prime))_\tau=-\frac{(\tilde{N}^2\cN\frf)^{-2j}A(\Gamma_\tau)^{1-j}(-j)!^2}{2(-2j)!}(2\sqrt{d_K})^{-j} \mathcal{M}_j({_\frf}\beta^\prime_\tau).
$$
Deninger shows that we obtain a similar result when considering the original torsion point ${_\frf}\beta$.  Indeed, by loc. cit. (2.6),
$$
\tilde{N}^{-4j}\mathcal{M}_j({_\frf}\beta_g^\prime)=\mathcal{M}_j({_\frf}\beta_g).
$$
For our purposes we must distinguish between the group $G_\frm$ and the principal homogeneous space of embeddings $\Hom_K(K(\frm),\bbC)$, so applying lemma \ref{eisenstein} we compute
\begin{align*}
&\rho_\infty(\xi_\frf(j))\\
&=\sum_{\tau\in\cT}(2\pi i)^j\left(-\frac{\cN\frf^{-2j}A(\Gamma_\tau)^{1-j}(-j)!^2}{2(-2j)!(2\sqrt{d_K})^{j}} \mathcal{M}_j(\Art(\rho_\frf)^{-1}{_\frf}\beta_\tau)\right)\cdot\tau\\
&=\sum_{g\in G_\frm}(2\pi i)^j\left(-\frac{\cN\frf^{-2j-1}A(\Gamma_g)^{1-j}(-j)!^2}{2(-2j)!(2\sqrt{d_K})^{j}} \mathcal{M}_j({_\frf}\beta_{\Art(\rho_\frf)^{-1}g})\right)\cdot g\tau_0\nonumber\\
&=\sum_{g\in G_\frm}(2\pi i)^j\left(-\frac{\cN\frf^{-2j-1}A(\Gamma_g)^{1-j}(-j)!^2}{2(-2j)!(2\sqrt{d_K})^{j}} \mathcal{M}_j({_\frf}\beta_g)\right)\cdot \Art(\rho_\frf)g\tau_0\nonumber\\
&=\sum_{g\in G_\frm}g^{-1}\Art(\rho_\frf)^{-1}(2\pi i)^j\left(-\frac{\cN\frf^{-2j-1}A(\Gamma_g)^{1-j}(-j)!^2}{2(-2j)!(2\sqrt{d_K})^{j}} \mathcal{M}_j({_\frf}\beta_g)\right)\cdot \tau_0\nonumber.
\end{align*}
The analysis over $\bbQ[G_\frm]$ is done character by character, so one projects to the $\chi$-isotypical component
\begin{multline}\label{rhoanal}
e_\chi(\rho_\infty(\xi_\frf(j))=\\ \left(\sum_{g\in G_\frm}-\frac{\cN\frf^{-2j-1}A(\Gamma_g)^{1-j}(-j)!^2}{2(-2j)!(2\sqrt{d_K})^{j}} \mathcal{M}_j({_\frf}\beta_g)\bar{\chi}(g)\bar\chi(\rho_\frf)\right)\cdot \eta_\bbQ
\end{multline}
where $\eta_\bbQ=e_\chi\cdot(2\pi i)^j\tau_0$ is the basis of $e_\chi\left(M_B^{+*}\right)$ determined by the choice of embedding. Comparing equation \eqref{rhoanal} with the formula in proposition \ref{denformula} we have the first part of the theorem.  A careful reader will note that there is a difference of a factor of $\cN\frf$ between our formula and Deninger's formula.  This is due to the fact that we scale $\rho_\infty$ by that factor in have agreement with \cite{HK1}.
\par
To deduce the norm compatibility, we first note that corestriction commutes with the regulator map, so it suffices to study the elements $w_\frf/w_{\frp\frf}\Tr_{K(\frp\frf)/K(\frf)}\rho_\infty(\xi_{\frf\frp}(j))$.  By the computation in \eqref{rhoanal} we have that
\begin{align*}
&\Tr_{K(\frp\frf)/K(\frf)}\rho_\infty(\xi_{\frf\frp}(j))\\
&=\begin{multlined}[t]\Tr_{K(\frp\frf)/K(\frf)}\sum_{g\in G_\frm}-g^{-1}\Art(\rho_{\frf\frp})^{-1}(2\pi i)^j\\\frac{\cN\frf\frp^{-2j}A(\Gamma_g)^{1-j}(-j)!^2}{2(-2j)!(2\sqrt{d_K})^{j}} \mathcal{M}_j({_{\frf\frp}}\beta_g)\cdot \tau_0\end{multlined}\\
&=\begin{multlined}[t]\sum_{g\in G_\frm}-g^{-1}\Art(\rho_{\frf\frp})^{-1}(2\pi i)^j\\\frac{\cN\frf\frp^{-2j}A(\Gamma_g)^{1-j}(-j)!^2}{2(-2j)!(2\sqrt{d_K})^{j}} \Tr_{K(\frp\frf)/K(\frf)}\mathcal{M}_j({_{\frf\frp}}\beta_g)\cdot \tau_0.\end{multlined}
\end{align*}
Focusing on $\cM_j$, we proceed by taking first the case of $\frp\mid\frf$.
\begin{align}
\Tr_{K(\frp\frf)/K(\frf)}\Art(\rho_\frp)^{-1}\mathcal{M}_j({_{\frf\frp}}\beta_g)=&\Tr_{K(\frp\frf)/K(\frf)}\Art(\rho_\frp)^{-1}\mathcal{M}_j({_{\frf\frp}}\beta_g)\nonumber\\
=&\Tr_{K(\frp\frf)/K(\frf)}\mathcal{M}_j({^{\Frob_\frp}}{_{\frf\frp}}\beta_{g})\nonumber\\
=& w_{\frp\frf}/w_\frf\sum_{u\in\Frob_\frp^{-1,*}{_\frf}\beta_g}\mathcal{M}_j({_\frf}\beta_g+u)\nonumber\\
=& w_{\frp\frf}/w_\frf\cN\frp^{2j}\mathcal{M}_j({_\frf}\beta_g)\label{divalpha}
\end{align}
Here the $u$ are the primitive $\frp$th roots of ${_\frp}\beta_g$ resulting from pulling back by the isogeny $E\stackrel{\Frob_\frp^{-1}}{\rightarrow}{^{\Frob_\frp}}E$, and the equality in \eqref{divalpha} follows from a formula in the proof of \cite[prop. 2.6]{D2}.  Now in the case that $\frp\nmid\frf$, there is a unique point $u_0$ which is not a primitive root and hence in $\Frob_\frp^{-1*}{_\frf}\beta_g$ but not a conjugate of any primitive root.  We add and subtract this point from the trace to conclude that
$$
\Art(\rho_{\frp})^{-1}\Tr_{K(\frp\frf)/K(\frf)}\mathcal{M}_j({_{\frf\frp}}\beta_g)=(1-\Frob_\frp^{-1})w_{\frp\frf}/w_\frf\cN\frp^{2j}\mathcal{M}_j({_\frf}\beta_g).
$$
Thus, we have shown that for a character $\chi$ of conductor $\frf$ 
$$
e_\chi(\rho_\infty(w_\frf/w_{\frp\frf}\Tr_{K(\frp\frf)/K(\frf)}\xi_{\frf\frp}(j))=
\begin{cases}
R_{\frf,j}L^\prime(\bar{\chi},j)\eta_\bbQ&\frp\mid\frf\\
&\\
(1-\bar{\chi}(\frp)\cN\frp^{-j})R_{\frf,j}L^\prime(\bar{\chi},j)\eta_\bbQ&\frp\nmid\frf,
\end{cases}
$$
where 
$$R_{\frf,j}=\frac{(-2\cN\frf)^{-1-j}\Phi(\frm)}{(-2j)!\Phi(\frf)}.$$
\end{proof}
\subsection{$\ell$-adic regulator}\label{ladic} We now study the image of the elements $\xi_\frf(j)$ under the \'etale Chern class map, which can be considered an $\ell$-adic regulator.  We begin this section with a brief review of the Euler system of elliptic units.  We use Kato's description of these elements, and refer to \cite{F3} for a comparison with more classical constructions.  
\begin{lemma}(\cite[15.4.4]{Ka} Let $E$ be an elliptic curve over a field $F$ with complex multiplication $\cO_K\cong\End_F(E)$ and let $\fra$ be an ideal of $\cO_K$ prime to $6$.
Then there is a unique function $${_\fra}\Theta_E\in \Gamma(E\setminus{_\fra}E,\cO^\times)$$ satisfying
\begin{itemize}
\item[(i)] $\text{div}({_\fra}\Theta_E)=N\fra\cdot(0)-E_\fra$, where $E_\fra$ denotes the $\fra$-torsion points of $E$.
\item[(ii)] For any $b\in\bbZ$ prime to $\fra$ we have $N_b({_\fra}\Theta_E)={_\fra}\Theta_E$ where $N_b$ is the norm map
associated to the finite flat morphism $E\setminus E_{b\fra}\to E\setminus E_\fra$ given by multiplication with $b$.
\end{itemize}
Moreover, for any isogeny $\phi:E\to E'$ of CM elliptic curves where $\End_F(E')=\cO_K$, we have
$\phi_*({_\fra}\Theta_E)={_\fra}\Theta_{E'}$, in particular property (ii) also holds with $b\in\cO_K$ prime to $\fra$.
\label{katotheta}\end{lemma}

Given $\frf\neq 1$ and any (auxiliary) $\fra$ which is prime to $6\frf$ we define an analog of the cyclotomic unit $1-\zeta_f$ by
$$
{_\fra}z_\frf={_\fra}\Theta_{\bbC/\frf}(1)
$$
and for $\frf=1$ we define a family of elements indexed by all ideals $\fra$ of $K$ by
$$ 
u(\fra)=\frac{\Delta(\cO_K)}{\Delta(\fra^{-1})}.
$$
where $\Delta(\tau)=q_\tau\prod(1-q_\tau^n)^{24}$ for $q_\tau=e^{2\pi i\tau}$ is the Ramanujan $\Delta$-function.
%
\begin{lemma} \cite[Lemma 2.2]{F3} The complex numbers ${_\fra}z_\frf$ and $u(\fra)$ satisfy the following properties
\begin{itemize}
\item[a)] (Integrality)
\begin{align*}&{_\fra}z_\frf\in\begin{cases}\cO_{K(\frf)}^\times & \text{$\frf$ divisible by primes $\frp\neq\frq$}\\
\cO_{K(\frf),\{v\mid\frf\}}^\times & \text{$\frf=\frp^n$ for some prime $\frp$}\end{cases}\\
&u(\fra)\cdot\cO_{K(1)}=\fra^{-12}\cO_{K(1)}
\end{align*}
\item[b)] (Galois action) For $(\frc,\frf\fra)=1$ with Artin symbol $\Art(\frc)\in\Gal(K(\frf)/K)$ we have
$$
{_\fra}z_\frf^{\Art(\frc)}={_\fra}z_{\frc^{-1}\frf};\quad u(\fra)^{\Art(\frc)}=u(\fra\frc)/u(\frc).
$$
This implies as in \ref{katotheta}
$$
{_\fra}z_\frf^{N\frc-\Art(\frc)}={_\frc}z_\frf^{N\fra-\Art(\fra)};\quad u(\fra)^{1-\Art(\frc)}=u(\frc)^{1-\Art(\fra)}.
$$
\item[c)] (Norm compatibility) For a prime ideal $\frp$ one has
$$N_{K(\frp\frf)/K(\frf)}({_\fra}z_{\frp\frf})^{w_\frf/w_{\frp\frf}}=
\begin{cases} 
{_\fra}z_\frf  & \frp\mid\frf\neq 1 \\
{_\fra}z_\frf^{1-\Frob_\frp^{-1}} & \frp\nmid\frf\neq 1\\
u(\frp)^{(\Art(\fra)-N\fra)/12} & \frf=1 
\end{cases}
$$
\item[e)] (Kronecker limit formula). Let $\eta$ be a complex character of $G_\frf$. If $\frf=1$ and $\eta\neq 1$ choose
any ideal $\fra$ so that $\eta(\fra)\neq 1$. Then
\begin{align*} L(\eta,0)=&\,\,\zeta_K(0)=-\frac{h}{w_1}R  &\eta=1\\
\frac{d}{ds}L(s,\eta)\vert_{s=0}=&-\frac{1}{1-\eta(\fra)} \frac{1}{12w_1}\sum_{\sigma\in G_1}\log |\sigma(u(\fra))|\eta(\sigma)
\quad\quad\eta\neq 1, &\frf=1\\
\frac{d}{ds}L(s,\eta)\mid_{s=0}=&-\frac{1}{N\fra-\eta(\fra)} \frac{1}{w_\frf}\sum_{\sigma\in G_\frf}\log |\sigma({_\fra}z_\frf)|\eta(\sigma)
&\frf\neq 1. \end{align*}
\end{itemize}
\label{ellipticunitproperties}
\end{lemma}
 \begin{remark}(i) The relations in b) show the auxiliary nature of $\fra$. (ii) The Galois action in b) together with the invariance under homothety shows that the Galois conjugates of ${_\fra}z_\frf$ are the numbers ${_\fra}\Theta_E(\alpha)$ where $(E,\alpha)$ runs through all pairs with $E/\bbC$ an elliptic curve and $\alpha\in E(\bbC)$ a primitive $\frf$-division point.
\end{remark}
We compute the image of $\xi_\frf(j)$ under the \'etale Chern class map $\rho_{et}$ in terms of elliptic units.
\begin{theorem}\label{ladicreg}   For all $1\neq\frf\mid\frm$, we have that
$$\rho_{et}(\xi_\frf(j))=\frac{\cN\frf^{-1-j}w_\frf}{(\cN\fra-\Art(\fra))(-2j)!\prod_{\frl\mid\ell}(1-\Frob^{-1}_\frl)}\cdot \left(\Tr_{K(\ell^n\frf)/K(\frf)}{_\fra}z_{\ell^n\frf}\zeta_{\ell^n}^{\otimes-j}\right)_n$$
up to a sign, where $\fra\nmid6\ell\frf$ is an auxilliary ideal and the ${_\fra}z_{\ell^n\frf}$ are elliptic units.
\end{theorem}
\begin{proof}  As $E$ is an abelian variety, the Todd classes vanish and the following diagram commutes.
$$
\xymatrix{
H^{1-2j}_\mathcal{M}(E^{-2j},1-2j)\ar[r]^{\rho_{et}}\ar[d]_{\mathcal{K}_\mathcal{M}}&H_{et}^{1-2j}(E^{-2j},\bbQ_\ell(1-2j))\ar[d]^{\mathcal{K}_\ell}\\
H^1_\mathcal{M}(\Spec(K(\frm)),1-j)\ar[r]_{\rho_{et}}&H_{et}^1(K(\frm),\bbQ_\ell(1-j))
}
$$
By \cite[Theorem 2.2.4]{HK1}, the \'etale realization of Eisenstein symbol can be computed in terms of the pullback of the elliptic polylogarithm sheaf along torsion sections. 
Indeed,
$$
\rho_{et}(\xi_\frf(j))=\mathcal{K}_\ell(\rho_{et}(\cE^{-2j}(\rho_\frf\cdot{_\frf}\beta)))=\Art(\rho_\frf)^{-1}\cdot\cN\frf^{-2j-1}\mathcal{K}_\ell({_\frf}\beta^*\cP\rm{ol}_{\bbQ_\ell})^{-2j}.  
$$
Happily, Kings computes this pullback up to a sign for an elliptic curve over any base \cite[Theorem 4.2.9]{K2} using the geometric elliptic polylog under the assumption that $\ell\nmid\frf$.  So we can now consider the action of the Kronecker map $\cK_\ell$ on 
\begin{equation}\label{kings4.2.9}
({_\frf}\beta^*\cP{\rm ol}_{\bbQ_\ell})^{-2j}=\frac{\pm\cN\frf^j}{\cN\fra([\fra]^{-2j}\cN\fra-1)(-2j)!}\left(\delta\sum_{[\ell^n]t_n={_\frf}\beta}{_\fra}\Theta_E(-t_n)(\tilde{t_n})^{\otimes -2j}\right)_n 
\end{equation}
which is an element of $H_{et}^{1-2j}(E^{-2j},\bbQ_\ell(1-2j))$.  Here, $\delta$ is the connecting homomorphism in a Kummer sequence, $\fra\subset\cO_K$ is chosen prime to $\ell\frf$, $[\fra]$ is the corresponding isogeny, and $\tilde{t}_n$ is a projection of $t_n\in E[\frf\ell^n]=E[\frf]\oplus E[\ell^n]$ to $E[\ell^n]$.  Kings gives the projection map as a composition
$$
E[\frf\ell^n]\stackrel{\Art(\rho_\frf)^{-1}}{\longrightarrow}{^{\Art(\rho_\frf)}}E[\ell^n]\stackrel{\Art(\rho_\frf)}{\longrightarrow}E[\ell^n],
$$ 
which accounts for the multiplication of his result by $\cN\frf^j$ above.  
\par
For a point $t\in E[\ell^n]$, we define $\gamma(t)^k:=<t,\sqrt{d_K}t>^{\otimes k}$ where $<,>$ is the Weil pairing.  Following section 5.1.1 of \cite{K2} we have that
$$
\mathcal{K}_\ell(\tilde{t}_n^{\otimes -2j})=\gamma(\tilde{t}_n)^{-j}=\zeta_{\ell^n}^{\otimes-j} \quad\text{and}\quad\mathcal{K}_\ell([\fra]^{-2j})=\cN\fra^{-j},
$$
where $\zeta_\ell^n=e^{\frac{2\pi i}{\ell^n}}$ is the canonical $\ell^n$th root of unity.
Note also that the Artin automorphism $\Art(\fra)$ acts on the space $H^1(K(\frm),\bbQ_\ell)$ via multiplication by $\cN\fra$, and thus on the space $H^1(K(\frm),\bbQ_\ell(1-j))$ by $\cN\fra^{2-j}$.  
We conclude that
$$
\rho_{et}(\xi_\frf(j))=\frac{\Art(\rho_\frf)^{-1}}{\cN\fra-\sigma(\fra)}\cdot\cN\frf^{-1-j}\left(\delta\sum_{[\ell^n]t_n={_\frf}\beta}{_\fra}\Theta_E(-t_n)\zeta_{\ell^n}^{\otimes-j}\right)_n.
$$
Now, it follows from the proof of lemma \ref{cohomologycomp} below that the Kummer map $\delta$ gives an isomorphism $\cO_{K(\frm)}[\frac{1}{\frm\ell}]^\times\otimes\bbQ_\ell(-j)\simeq H^1(K(\frm),\bbQ_\ell(1-j))$.  In the sequel, we will drop the map $\delta$ from our formulas and consider the equalities as occurring inside of the unit group.
As we consider imaginary quadratic fields with any class number, we need the following analog of \cite[Lemma 5.1.2]{K2}.  
\begin{lemma}\label{trace}
For any rational prime $\ell$,
\begin{multline*}
\prod_{\frl\mid\ell}(1-\Frob^{-1}_\frl)^{-1}\left(\sum_{[\ell^n]t_n= \Omega f_\frf^{-1}}{_\fra}\Theta_E(-t_n)\otimes\zeta_{\ell^n}^{\otimes-j}\right)_n=\\w_\frf\left(\Tr_{K(\ell^n\frf)/K(\frf)}{_\fra}\Theta_E(-s_n)\otimes\zeta_{\ell^n}^{\otimes-j}\right)_n,
\end{multline*}
where $s_n$ is a primitive $\ell^n$th root of ${_\frf}\beta$.
\end{lemma}
\begin{proof}
Let $\frl$ be a prime of $K$ and $\nu=\mathrm{ord}_\frl(\frf)$.  Define $t_r$ via the main theorem of CM so that $\Frob_\frl^{-r}t_r=\Omega f_\frf^{-1}$ and write $t_r=(\tilde{t}_r,t_{r,0})\in E[\frl^{r+\nu}]\oplus E[\frf_0]=E[\frl^r\frf_0]$.  We also abuse notation and write $\frl^it_r$ for $\Frob_\frl^{-i}t_r$.  Define a filtration $F^\bullet$ on the set $H^\frl_r=\{\frl^rt_r=\Omega f_\frf^{-1}\}$ by
$$
F^i_r:=\{t_r=(\tilde{t}_r,t_{r,0})\in H^\frl_{r,t}:\frl^{r+\nu-i}\tilde{t}_r=0\}.
$$
The Frobenius at $\frl$ acts via $(\Frob^{-1}_\frl)\zeta_{\ell^r}^{\otimes k}=\zeta_{\ell^{r-1}}^{\otimes k}$ and fixes $\Tr_{K(\frl^r\frf)/K(\frl^{r-i}\frf)}{_\fra}\Theta_E(-s_r)$ as the Galois group is generated by the frobenius.  Thus, we compute
\begin{align*}
\Frob_\frl^{-i}\Tr_{K(\frl^r\frf)/K(\frl^{r-i}\frf)}{_\fra}&\Theta_E(-s_r)\otimes\zeta_{\ell^r}^{\otimes-j}\\&=\Tr_{K(\frl^r\frf)/K(\frl^{r-i}\frf)}{_\fra}\Theta_E(-(\tilde{s}_r,s_{r,0}))\otimes\zeta_{\ell^{r-i}}^{\otimes-j}\\
&={_\fra}\Theta_E(-(\tilde{s}_{r-i},s_{r-i,0}))\otimes\zeta_{\ell^{r-i}}^{\otimes-j}.
\end{align*}
The second equality follows from the distribution relation for elliptic units in lemma \ref{ellipticunitproperties}.  Notice that the elliptic function ${_\fra}\Theta_E$ does not change in the distribution relation even though the curve does because the lattices are homothetic.
\par
The Galois group $\Gal(K(\frl^{r-i}\frf)/K(\frf))$ acts transitively on $F^i_r\setminus F^{i+1}_r$ with each conjugate appearing $w_\frf$ times.  Hence we can write
$$
\Frob_\frl^{-i}\Tr_{K(\frl^r\frf)/K(\frf)}{_\fra}\Theta_E(-s_r)\otimes\zeta_{\ell^r}^{\otimes-j}=
\frac{1}{w_\frf}\sum_{t_{r-i}\in F^i_r\setminus F^{i+1}_r}{_\fra}\Theta_E(-(\tilde{t}_{r-i},t_{r-i,0}))\otimes\zeta_{\ell^{r-i}}^{\otimes-j}
$$
These elements are annihilated by $\frl^r$, so summing over $i$ we can take the limit as $r\rightarrow\infty$ to get
\begin{align*}
\big(\sum_{\frl^rt_r={_\frf}\beta}{_\fra}&\Theta_E(-t_r)\otimes\zeta_{\ell^r}^{\otimes-j}\big)_r\\&=w_\frf\left(\sum_{i=1}^r(\Frob_\frl^{-1})^i\Tr_{K(\frl^r\frf)/K(\frf)}{_\fra}\Theta_E(-s_r)\otimes\zeta_{\ell^r}^{\otimes-j}\right)_r\\
&=w_\frf(1-\Frob_\frl^{-1})^{-1}\left(\Tr_{K(\frl^r\frf)/K(\frf)}{_\fra}\Theta_E(-s_r)\otimes\zeta_{\ell^r}^{\otimes-j}\right)_r.
\end{align*}
\noindent
For $\ell$ inert in $K$, the lemma is proved, and for $\ell$ split or ramified in $K$ we apply the results to $\Tr_{K(\ell^n\frf)/K(\frf)}=\Tr_{K(\ell^n\frf)/K(\frl^n\frf)}\Tr_{K(\frl^n\frf)/K(\frf)}$.\end{proof}

Again, by the main theorem of complex multiplication, $\Art(\rho_\frf)^{-1}\cdot s_n$ gives a primitive torsion point of $1\bmod\frf$ on the curve ${^{\Art(\rho_\frf)}}E$ with $\bbC/\Omega\frf\simeq{^{\Art(\rho_\frf)}}E(\bbC)$.  Therefore, we effectively undo our choice of ${_\frf}\beta$ via the identity
$$
\Art(\rho_\frf)^{-1}{_\fra}\Theta_E(-s_n)={_\fra}z_{\frf\ell^n}.
$$
 In particular, we have shown that the $\chi$ component is given by
\begin{equation}\label{ladicregeq}
e_\chi\cdot\rho_{et}(\xi_{\frf}(j))=\prod_{\frl\mid\ell}(1-\chi(\frl)\cN\frl^{-j})^{-1}\cN\frf^{-1-j}\frac{w_\frf}{(-2j)!}\left(\Tr_{K(\ell^n\frf)/K(\frf)}z_{\ell^n\frf}\zeta_{\ell^n}^{-j}\right)_n.
\end{equation}
where we follow Kato to set $z_{\ell^n\frf}=(\cN\fra-\sigma(\fra))^{-1}{_\fra}z_{\ell^n\frf}$.  This completes the proof of theorem \ref{ladicreg}.  \end{proof}
\subsection{Putting it all together}
We can now compute the image of the special value $L^*(_AM)$ under the composition $\vartheta_\ell\vartheta_\infty$.  We begin by studying the complex which computes the cohomology groups of interest. Let $\mu=\mathrm{ord}_\ell\frm$ be compound notation denoting the following
$$
\ell^\mu=\begin{cases}
\frl_1^{\mu_1}\frl_2^{\mu_2} & \ell=\frl_1\frl_2 \mbox{  split}\\
\frl_1^{\mu_1} &\ell=\frl_1^2 \mbox{  ramified}\\
\ell^{\mu}& \ell \mbox{  inert},
\end{cases}
$$
where $\mu_1, \mu_2\in\bbZ$, and we write $\frm=\frm_0\ell^\mu$.  
We choose a projective $G_\bbQ$-stable $\bbZ_\ell[G_\frm]$ lattice
$$
T^\prime_\ell=H^0_{et}(\Spec(K(\frm)\otimes_K\overline{K}), \bbZ_\ell)=T_\ell(-j)
$$
 in the $\ell$-adic realization, 
 $$
 M_\ell(-j)=H^0_{et}(\Spec(K(\frm)\otimes_K\overline{K}), \bbQ_\ell),
 $$  
define a perfect complex of $\bbZ_\ell[G_\frm]$-modules, 
\begin{equation}\label{deltakm}
\Delta(K(\frm)):=R\Gamma(\cO_K[\frac{1}{\frm\ell}], T^\prime_\ell(1))
\end{equation}

For any finite set of places $S$ of $K(\frm)$, the $\bbZ[G_\frm]$ module $X_S$ is defined to be the kernel of the sum map
$$
0\rightarrow X_S(K(\frm))\rightarrow Y_S(K(\frm))\rightarrow \bbZ\rightarrow0
$$
where  $Y_S(K(\frm)):=\bigoplus_{v\in S}\bbZ$.  When there is no confusion we will suppress the field.
\begin{lemma}\label{cohomologycomp}
The cohomology of $\Delta(K(\frm))$ is given by a canonical isomorphism, 
$$
H^1(\Delta(K(\frm))\simeq H^1(\cO_{K(\frm)}[\frac{1}{\frm\ell}],\bbZ_\ell(1))\simeq\cO_{K(\frm)}[\frac{1}{\frm\ell}]^\times\otimes_\bbZ\bbZ_\ell,
$$
a short exact sequence,
$$
0\rightarrow\Pic(\cO_{K(\frm)}[\frac{1}{\frm\ell}])\otimes_\bbZ\bbZ_\ell\rightarrow H^2(\Delta(K(\frm)))\rightarrow X_{\{v\mid \frm \ell\infty\}}\otimes_\bbZ\bbZ_\ell\rightarrow0,
$$
and $H^i(\Delta(K(\frm)))=0$ for $i\neq1,2$.
\end{lemma}
\begin{proof}By Shapiro's lemma, 
$$
R\Gamma(\cO_K[\frac{1}{\frm\ell}],T^\prime_\ell(1))\simeq R\Gamma(\cO_{K(\frm)}[\frac{1}{\frm\ell}],\bbZ_\ell(1)).
$$

The Kummer sequence
$$
0\rightarrow\mu_{\ell^n}\rightarrow\mathbb{G}_m\stackrel{\ell^n}{\rightarrow}\mathbb{G}_m\rightarrow0
$$
induces the long exact cohomology sequence 
$$
\stackrel{\ell^n}{\rightarrow} H^i(\cO_{K(\frm)}[\frac{1}{\frm\ell}],\mathbb{G}_m)\rightarrow H^{i+1}(\cO_{K(\frm)}[\frac{1}{\frm\ell}],\mu_{\ell^n})\rightarrow H^{i+1}(\cO_{K(\frm)}[\frac{1}{\frm\ell}],\mathbb{G}_m)\stackrel{\ell^n}{\rightarrow} .
$$
The Galois cohomology is then computed by the short exact sequences
\begin{align*}
0\rightarrow H^0(\cO_{K(\frm)}[\frac{1}{\frm\ell}],\mathbb{G}_m)/\ell^n\rightarrow &H^{1}(\cO_{K(\frm)}[\frac{1}{\frm\ell}],\mu_{\ell^n})\\
&\rightarrow H^{1}(\cO_{K(\frm)}[\frac{1}{\frm\ell}],\mathbb{G}_m)[\ell^n]\rightarrow0\\
0\rightarrow H^1(\cO_{K(\frm)}[\frac{1}{\frm\ell}],\mathbb{G}_m)/\ell^n\rightarrow &H^{2}(\cO_{K(\frm)}[\frac{1}{\frm\ell}],\mu_{\ell^n})\\
&\rightarrow H^{2}(\cO_{K(\frm)}[\frac{1}{\frm\ell}],\mathbb{G}_m)[\ell^n]\rightarrow0
\end{align*}
and the canonical isomorphism
$$
H^0(\cO_{K(\frm)}[\frac{1}{\frm\ell}],\mu_{\ell^n})\simeq H^0(\cO_{K(\frm)}[\frac{1}{\frm\ell}],\mathbb{G}_m)[\ell^n].
$$
The Galois cohomology of $\cO_{K(\frm)}[\frac{1}{\frm\ell}]^\times$ is given by \cite[Prop. 8.3.10]{NSW}
$$
H^i(\cO_{K(\frm)}[\frac{1}{\frm\ell}],\mathbb{G}_m)=
\begin{cases}
\cO_{K(\frm)}[\frac{1}{\frm\ell}]^\times&i=0\\
\Pic(\cO_{K(\frm)}[\frac{1}{\frm\ell}])&i=1,\\
\end{cases}
$$
and $H^3(\cO_{K(\frm)}[\frac{1}{\frm\ell}], \mathbb{G}_m)(\ell)=0$.
Moreover, $H^2(\cO_{K(\frm)}[\frac{1}{\frm\ell}], \mathbb{G}_m)$ is $\ell$-divisible (loc. cit. Corollary 8.3.11).
Hence, taking inverse limits we compute that 
$$
H^i(\cO_{K(\frm)}[\frac{1}{\frm\ell}],\bbZ_\ell(1))=
\begin{cases}
\cO_{K(\frm)}[\frac{1}{\frm\ell}]^\times\otimes_\bbZ\bbZ_\ell&i=1\\
\Pic(\cO_{K(\frm)}[\frac{1}{\frm\ell}])\otimes_\bbZ\bbZ_\ell&i=2\\
0&\mbox{otherwise.}
\end{cases}
$$
\end{proof}
\par
\begin{remark} A similar computation is given in \cite[Prop 3.3]{BF2} under the condition that the $S$-restricted class group is trivial.  In our case, since $K$ is a totally imaginary field, the Tate cohomology is just the usual cohomology with compact supports.  Thus the definition of $\Delta(K(\frm))$ is simplified since Artin-Verdier duality gives
\begin{equation}\label{artinverdier}
R\Hom_{\bbZ_\ell}(R\Gamma_c(\cO_{K(\frm)}[\frac{1}{\frm\ell}],\bbZ_\ell),\bbZ_\ell)[-3]\simeq R\Gamma(\cO_{K(\frm)}[\frac{1}{\frm\ell}],\bbZ_\ell(1)).
\end{equation}
\end{remark}
For invertible $\bbZ_\ell[G_\frm]$-modules, the dual of the inverse module (or vice versa) is isomorphic to the original module with the action of $G_\frm$ twisted by the automorphism $g\mapsto g^{-1}$.  We denote the twisted action with a $\#$.  Hence, there is a natural isomorphism of determinants,
$$
\Det_{\bbZ_\ell[G_\frm]}\Delta(K(\frm))\simeq\Det_{\bbZ_\ell[G_\frm]}R\Gamma_c(\cO_K[\frac{1}{\frm\ell}], T^\prime_\ell)^\#.
$$
\begin{theorem}\label{lvaluebasis}
The element ${_A}\vartheta_\ell({_A}\vartheta_\infty(L^*({_A}M,0)^{-1}))^\#$ of 
$$\Det_{A_\ell}\Delta(K(\frm))=\prod_{\chi\in\hat{G}}(\Det_{\bbQ_\ell(\chi)}\Delta(K(\frm))\otimes\bbQ_\ell(\chi))$$
has $\chi$ component given by

$$
R_{\chi,j}\left(\Tr_{K(\frf_{\chi,0}\ell^n)/K(\frf_\chi)}({_\fra}z_{\frf_{\chi,0}\ell^n}\zeta_{\ell^n}^{\otimes-j})\right)^{-1}_n\otimes\zeta_{\ell^\infty}^{\otimes-j}\cdot e_{\chi}\tau_0,
$$
where $R_{\chi,j}=\prod_{\frp\mid\frm_0}(1-\chi(\frp)\cN\frp^{-j})^{-1}\frac{[K(\frm):K(\frf_\chi)]}{(-2)^{1+j}}(\cN\fra-\chi(\fra)\cN\fra^{-j})$ and $\zeta_{\ell^\infty}=(\zeta_{\ell^n})_n$.
\end{theorem}

\begin{proof}
The dual of the regulator isomorphism
$$
\rho_\infty^\vee:H^0_B(K(\frm)(\bbC),\bbQ(j))^+\otimes_\bbQ\bbR\stackrel{\sim}{\rightarrow} K_{1-2j}(\cO_{K(\frm)})^*\otimes_\bbZ\bbR 
$$
induces an isomorphism of rank 1 $A\otimes\bbR$-modules 
$${_A}\vartheta_\infty:A\otimes_\bbQ\bbR\rightarrow\Xi({_A}M)\otimes_\bbQ\bbR,$$
 where we recall that 
\begin{align*}
\Xi({_A}M)=(K_{1-2j}(\cO_{K(\frm)})^*\otimes_\bbZ\bbQ)\otimes H^0_B(K(\frm)(\bbC),\bbQ(j))^+.
\end{align*}
\par
In Theorem \ref{motivic} we proved that for $\frf_\chi\neq1$,
$$
e_\chi(\rho_\infty(\xi_{\frf_\chi}(j)))=\frac{(-2\cN\frf)^{-1-j}\Phi(\frm)}{(-2j)!\Phi(\frf)}L^\prime(\bar{\chi},j)\eta_\bbQ,
$$
where $\eta_\bbQ$ is a basis of $e_\chi(M_B^{+*})$.  Moreover, for $\frf_\chi=1$, we have a computation via trace maps
$$
e_\chi\cdot\Tr_{K(\frq)/K(1)}\rho_\infty(w_K(1-\Frob_\frq^{-1})^{-1}\xi_{\frq}(j))=\frac{(-2)^{-1-j}\Phi(\frm)}{(-2j)!}L^\prime(\bar{\chi},j)\eta_\bbQ,
$$
where we take the primitive $L$-function for $\chi$.  As corestriction commutes with the regulator maps, we will sometimes abuse notation at write $\xi_1(j)$ for a choice of $\Tr_{K(\frq)/K(1)}{_\frq}\xi_1(j)$.
\par
Since both $H^0_B(K(\frm)(\bbC),\bbQ(j))^+$ and $K_{1-2j}(\cO_{K(\frm)})\otimes_\bbZ\bbQ$ are invertible $A$-modules duality manifests in terms of the twist $g\mapsto g^{-1}$ according to the computation
\begin{align}
\Xi({_A}M)^\#&=(K_{1-2j}(\cO_{K(\frm)})^*\otimes_\bbZ\bbQ)^\#\otimes (H^0_B(K(\frm)(\bbC),\bbQ(j))^{+,-1})^\#\nonumber\\
                        &=(K_{1-2j}(\cO_{K(\frm)})\otimes_\bbZ\bbQ)^{-1}\otimes (H^0_B(K(\frm)(\bbC),\bbQ(j))^*)^+\nonumber\\
                        &=(K_{1-2j}(\cO_{K(\frm)})\otimes_\bbZ\bbQ)^{-1}\otimes Y(-j)\label{twistisom},
\end{align}
where for $v$ a place of $K(\frm)$
$$Y(-j):=\bigoplus_{v\mid\infty}\bbQ\cdot(2\pi i)^{-j}.$$
The $\Gal(\bbC/\bbR)$-equivariant perfect pairing
$$
\bigoplus_{\tau\in\cT}\bbR\cdot(2\pi i)^j\times\bigoplus_{\tau\in\cT}\bbC/\bbR\cdot(2\pi i)^{1-j}\rightarrow\bigoplus_{\tau\in\cT}\bbC/2\pi i\cdot\bbR\stackrel{\Sigma}{\rightarrow}\bbR
$$
for $\cT=\Hom(K(\frm),\bbC)$ identifies the $\bbQ$-dual of $H^0_B(K(\frm)(\bbC),\bbQ(j))$ with $\bigoplus_{\tau\in\cT}\bbQ\cdot(2\pi i)^{-j}$.  Taking invariants under complex conjugation gives the equality in \eqref{twistisom}.
We compute that the $\chi$ components of ${_A}\vartheta_\infty^\#(L^*({_A}M,0)^{-1})=(L^*({_A}M,0)^{-1})^\#{_A}\vartheta_\infty(1)$ are given by
$$
({_A}\vartheta_\infty^\#(L^*({_A}M,0)^{-1}))_\chi=
\frac{(-2\cN\frf)^{-1-j}\Phi(\frm)}{(-2j)!\Phi(\frf)}[\xi_{\frf_\chi}(j)]^{-1}\otimes(2\pi i)^{-j}e_\chi\tau_0.
$$
Denote by $\Delta(K(\frm))_j$ the ``twist" of the $\bbZ_\ell[G_\frm]$-module $\Delta(K(\frm))$.  Namely,
$$
\Delta(K(\frm))_j:=(R\Gamma(\cO_K[\frac{1}{\frm\ell}],T_\ell).
$$ 
 The natural isomorphism
\begin{align*}
\Det_{\bbZ_\ell[G]}\Delta(K(\frm))&=(\Det_{\bbZ_\ell[G]}R\Gamma_c(\cO_K[\frac{1}{\frm\ell}],T^\prime_\ell)^*)^{-1}\\
                                                  &\simeq\Det_{\bbZ_\ell[G]}R\Gamma_c(\cO_K[\frac{1}{\frm\ell}],T^\prime_\ell)^\#
\end{align*}
induces
$$
\Det_{\bbZ_\ell[G]}\Delta(K(\frm))_j\simeq\Det_{\bbZ_\ell[G]}R\Gamma_c(\cO_K[\frac{1}{\frm\ell}],T_\ell)^\#.
$$
By lemma \ref{cohomologycomp} there are isomorphisms in cohomology
\begin{align*}
H^1(\Delta(K(\frm))_j)\otimes_{\bbZ_\ell}\bbQ_\ell&\simeq H^1(\cO_{K(\frm)}[\frac{1}{\frm\ell}],\bbQ_\ell(1-j))\\
H^2(\Delta(K(\frm))_j)\otimes_{\bbZ_\ell}\bbQ_\ell&\simeq\left(\bigoplus_{\tau\in\cT}\bbQ_\ell(-j)\right)^+
\end{align*}
with $H^i(\Delta(K(\frm))_j)=0$ for $i\neq1,2$.
\par
Thus far, we have shown for $\frf_\chi\neq1$,
\begin{multline*}({_A}\vartheta_\ell\circ{_A}\vartheta_\infty(L^*({_A}M,0)^{-1}))_\chi=\\
\prod_{\frp\mid\frm\ell}(1-\chi(\frp)\cN\frp^{-j})^{-1}\frac{(-2\cN\frf)^{-1-j}\Phi(\frm)}{(-2j)!\Phi(\frf)}\rho_{et}(\xi_{\frf_\chi}(j))^{-1}\otimes\zeta_{\ell^\infty}^{\otimes-j}\cdot\sigma, 
\end{multline*}
and for $\frf_\chi=1$, we choose a $\frq\mid\frm$ to show
\begin{multline}
({_A}\vartheta_\ell\circ{_A}\vartheta_\infty(L^*({_A}M,0)^{-1}))_\chi=\\
\prod_{\frq\neq\frp\mid\frm\ell}(1-\chi(\frp)\cN\frp^{-j})^{-1}(-2)^{1+j}\Phi(\frm)\rho_{et}(w_K\xi_{1}(j))^{-1}\otimes\zeta_{\ell^\infty}^{\otimes-j}\cdot\sigma.
\end{multline}
Theorem \ref{ladicreg} states that for any $1\neq\frf\mid\frm$,
$$
\rho_{et}(\xi_\frf(j))=\frac{\cN\frf^{-1-j}w_\frf}{(\cN\fra-\sigma(\fra))\prod_{\frl\mid\ell}(-2j)!(1-\Frob^{-1}_\frl)}\cdot\left(\Tr_{K(\ell^n\frf)/K(\frf)}{_\fra}z_{\ell^n\frf}\zeta_{\ell^n}^{\otimes-j}\right)_n.
$$
We recall that $[K(\frf):K(1)]=\Phi(\frf)w_\frf/w_K$  where $w_K\in\{2,4,6\}$ is the number of roots of unity in the imaginary quadratic field $K$, and $w_\frf$ is the number of roots of unity in $K$ which are congruent to 1 modulo $\frf$.  For $\frf$ large enough (at least bigger than 2) this number is 1.  Recall that we have chosen $\frm$ so that $w_\frm=1$, so we have that $\Phi(\frm)/\Phi(\frf_\chi)=[K(\frm):K(\frf_\chi)]w_{\frf_\chi}$.  What's more, if $(\ell,\frf)\neq1$, then by lemma \ref{ellipticunitproperties},
\begin{align*}
\left(\Tr_{K(\ell^n\frf)/K(\frf)}{_\fra}z_{\ell^n\frf}\zeta_{\ell^n}^{\otimes-j}\right)_n&=\left(\Tr_{K(\ell^n\frf_0)/K(\frf)}\left(\Tr_{K(\ell^{n+\mu}\frf_0)/K(\ell^n\frf_0)}{_\fra}z_{\ell^n\frf}\right)\zeta_{\ell^n}^{\otimes-j}\right)_n\\
&=\left(\Tr_{K(\ell^n\frf_0)/K(\frf)}{_\fra}z_{\ell^n\frf_0}\zeta_{\ell^n}^{\otimes-j}\right)_n
\end{align*}
where $\mu$ denotes the compound notation discussed above.  Thus for $\frf_\chi\neq1$ we have computed the $\chi$-component of the image of the $L$-value.
When $\frf_\chi=1$, choose $\frq$ so that $w_\frq=1$ and compute
\begin{align*}
\rho_{et}(w_K&\Tr_{K(\frq)/K(1)}\xi_\frq(j))\\&=w_K\frac{1}{(\cN\fra-\Art(\fra))\prod_{\frl\mid\ell}(1-\Frob_\frl^{-1})}\cdot(\Tr_{K(\ell^n\frq)/K(1)}{_\fra}z_{\ell^n\frq}\zeta_{\ell^n}^{\otimes-j})_n\\
&=\frac{(1-\Frob_\frq^{-1})}{(\cN\fra-\Art(\fra))\prod_{\frl\mid\ell}(1-\Frob_\frl^{-1})}\cdot(\Tr_{K(\ell^n)/K(1)}{_\fra}z_{\ell^n}\zeta_{\ell^n}^{\otimes-j})_n.
\end{align*}
Substituting the formulas for $\rho_{et}$ completes the proof of the theorem.
\end{proof}
\section{Descent from the main conjecture of Iwasawa theory}\label{Iwasawa}
\subsection{Formulation of the conjecture}
We first formulate the 2-variable main conjecture by considering the tower of ray class fields over $K(\frm)$ unramified outside of the primes above $\ell$.  The Iwasawa algebra
$$
\Lambda:=\varprojlim_{n}\bbZ_\ell[G_{\frm\ell^n}]\simeq\bbZ_\ell[G^{tor}_{\frm\ell^\infty}][[S,T]]
$$
is a finite product of complete local 3-dimensional Cohen-Macaulay rings, where $G^{tor}_{\frm\ell^\infty}$ is the torsion subgroup of $G_{\frm\ell^\infty}=\varprojlim_nG_{\frm\ell^n}$.  $\Lambda$ is regular if and only if $\ell\nmid\#G^{tor}_{\frm\ell^\infty}$.
In general, this torsion subgroup is not $G_{\frm_0\ell}$ where $\frm_0$ is the prime to $\ell$ part of $\frm$.  (Consider the case that $\ell\mid h_K$.)
The elements $S,T\in\Lambda$ depend on the choice of a complement $F\simeq\bbZ_\ell^2$ of the torsion subgroup in $G_{\frm\ell^\infty}$ as well as the choice topological generators $\gamma_1,\gamma_2$ of $F$.
The cohomology of the perfect complex of $\Lambda$ modules,
$$
\Delta^\infty=\varprojlim_n\Delta(K(\frm\ell^n))
$$
is computed by functoriality.  By Lemma \ref{cohomologycomp}, $H^i(\Delta^\infty)=0$ for $i\neq1,2$, and we have a canonical isomorphism,
$$
H^1(\Delta^\infty)\simeq U^\infty_{\{v\mid\frm\ell\}}:=\varprojlim_n\cO_{K(\frm_0\ell^n)}[\frac{1}{\frm\ell}]^\times\otimes_\bbZ\bbZ_\ell,
$$
and a short exact sequence,
$$
0\rightarrow P^\infty_{\{v\mid\frm\ell\}}\rightarrow H^2(\Delta^\infty)\rightarrow X^\infty_{\{v\mid\frm\ell\infty\}}\rightarrow0,
$$
where 
\begin{align*}
P^\infty_{\{v\mid\frm\ell\}}&:=\varprojlim_n\Pic(\cO_{K(\frm\ell^n)}[\frac{1}{\frm\ell}])\otimes_\bbZ\bbZ_\ell\\
X^\infty_{\{v\mid\frm\ell\infty\}}&:=\varprojlim_nX_{\{v\mid\frm\ell\infty\}}(K(\frm\ell^n))\otimes_\bbZ\bbZ_\ell.
\end{align*}
The limits are taken with respect to the Norm maps, which on the module $Y_S$ is the map sending a place to its restriction.  We also consider $K(\frm\ell^n)$ as a subfield of $\bbC$ and denote the corresponding archimedean place by $\sigma_{\frm\ell^n}$.  Notice that for $\frf_0\mid\frm_0$, the elliptic units ${_\fra}z_{\frf_0\ell^n}$ discussed in section \ref{ladic} form a Norm-compatible system of units.  We set
\begin{align*}
{_\fra}\eta_{\frf_0}&:=({_\fra}z_{\frf_0\ell^n})_{n>>0}\in U^\infty_{\{v\mid\frm\ell\}}\\
\sigma&:=(\sigma_{\frm\ell^n})_{n>>0}\in Y^\infty_{\{v\mid\frm\ell\infty\}}
\end{align*}  
We fix an embedding $\bar{\bbQ}_\ell\to\bbC$ and identify $\hat{G}$ with the set of $\bar{\bbQ}_\ell$-valued characters. The total ring of fractions
\begin{equation}
Q(\Lambda)\cong\prod_{\psi\in(\hat{G}_{\frm\ell^\infty}^{tor})^{ \bbQ_\ell}}Q(\psi)\label{qdecomp}
\end{equation} 
of $\Lambda$ is a product of fields indexed by the $\bbQ_\ell$-rational characters of $G_{\frm \ell^\infty}^{tor}$.  Since for any place $w$ of $K$, the $\bbZ[G_{\frm\ell^n}]$-module $Y_{\{v\mid w\}}(K(\frm\ell^n))$ is induced from the trivial module $\bbZ$ on the decomposition group $D_w\subseteq G_{\frm\ell^n}$, and for $w=\infty$ (resp. nonarchimedean $w$) we have $[G_{\frm\ell^n}:D_w]=[K(\frm\ell^n): K]$ (resp. the index $[G_{\frm\ell^n}:D_w]$ is bounded as $n\to\infty$), one computes easily
\begin{equation}
 \dim_{Q(\psi)}(Y^\infty_{\{v\mid m\ell\infty\}}\otimes_\Lambda Q(\psi))= 1\label{rankY}
\end{equation}
for all characters $\psi$.  Note that the inclusion $X^\infty_{\{v\mid \frm\ell\infty\}}\subseteq Y^\infty_{\{v\mid \frm\ell\infty\}}$ becomes an isomorphism after tensoring with $Q(\psi)$, and thus by the unit theorem
\begin{equation}
\dim_{Q(\psi)}(U^\infty_{\{v\mid\frm\ell\}}\otimes_\Lambda Q(\psi))=1\label{rankU}.
\end{equation} 
So we have that $e_\psi({_\fra}\eta_{\frm_0}^{-1}\otimes\sigma)$ is a $Q(\psi)$-basis of
\begin{align*}
&\Det_{Q(\psi)}^{-1}(U^\infty_{\{v\mid \frm \ell\}}\otimes_\Lambda Q(\psi))\otimes\Det_{Q(\psi)}(X^\infty_{\{v\mid\frm\ell\infty\}}\otimes_\Lambda Q(\psi)\\
\cong &\Det_{Q(\psi)}\left(\Delta^\infty\otimes_\Lambda Q(\psi)\right). 
\end{align*}
The last isomorphism follows from the fact that the class group, $P^\infty_{\{v\mid\frm \ell\}}$ is a torsion $\Lambda$-module. 
Hence we obtain an element
$$
\mathcal L := (\cN\fra-\Art(\fra)){_\fra}\eta_{\frm_0}^{-1}\otimes\sigma\in\Det_{Q(\Lambda)}\left(\Delta^\infty\otimes_\Lambda
Q(\Lambda)\right).
$$
\begin{conj3}
There is an equality of invertible $\Lambda$-submodules
$$
\Lambda\cdot\mathcal L = \Det_{\Lambda}\Delta^\infty
$$
of $\Det_{Q(\Lambda)}\left(\Delta^\infty\otimes_\Lambda Q(\Lambda)\right)$. 
\end{conj3}
\begin{theorem} \cite[Theorem 5.7]{JLK}
 The Iwasawa main conjecture holds, for all prime number $\ell\nmid6$ which are split in $K$.
\end{theorem}
\begin{remark} i) In order to prove the theorem, it is necessary to show that the $\mu$-invariant of a certain Iwasawa module vanishes.  This follows from a result of Gillard \cite[3.4]{Gi} when $\ell\nmid 6$ is split in $K$.  If one were to prove that the $\mu$-invariant vanishes for non-split $\ell$, the Iwasawa main conjecture would follow immediately. \\
 ii) The statement of the theorem in \cite{JLK} can be rewritten to coincide with the conjecture above when one notes that
$$
\Delta^\infty=R\Gamma(\cO_K[\frac{1}{\frm\ell}],\Lambda(1)).
$$
\end{remark}
\subsection{Descent and proof of the main theorem}
\begin{theorem}\label{descent}
The Iwasawa main conjecture implies the local equivariant Tamagawa number conjecture for the pair $(K(\frm),G_\frm)$ when $j<0$ for every prime $p\neq 2$.
\end{theorem}
To prove this theorem we will show that the equality of $\Lambda$-modules in the Iwasawa main conjecture descends to
\begin{align*}
{_A}\vartheta_\ell\circ{_A}\vartheta_\infty^\#(L^*({_A}M,0)^{-1})\cdot\bbZ_\ell[G_\frm]=\Det_{\bbZ_\ell[G_\frm]}\Delta(K(\frm))
\end{align*}
in $\Det_{\bbQ_\ell[G_\frm]}(\Delta(K(\frm))\otimes_{\bbZ_\ell}\bbQ_\ell)$.  We begin by proving a twisting lemma.   For $j\in\bbZ$ we denote by $\kappa^j:G_{\frm \ell^\infty}\to\Lambda^\times$ the character $g\mapsto\chi_{\mathrm{cyclo}}(g)^jg$ as well as the induced ring automorphism $\kappa^j:\Lambda\to\Lambda$. If there is no risk of
confusion we also denote by $\kappa^j:\Lambda\to\bbZ_\ell[G_\frm]\subseteq A_\ell$ the composite of $\kappa^j$ and the natural projection to $\bbZ_\ell[G_\frm]$ or $A_\ell$.  
\begin{lemma}a) For $j\in\bbZ$ there is a natural isomorphism
$$
\Delta^\infty\otimes_{\Lambda,\kappa^j}^{\mathbb L}\bbZ_\ell[G_\frm]\rightarrow\Delta(K(\frm))_j.
$$
b) On the cohomology groups, the map $H^i(\Delta^\infty)\rightarrow H^i(\Delta^\infty_j)$ induces
\begin{align*}
u\mapsto(u_n\cup\zeta_{\ell^n}^{\otimes-j})_{n>>0}\mbox{  and  }s\mapsto(s_n\cup\zeta_{\ell^n}^{\otimes-j})_{n\geq0}
\end{align*}
where
\begin{align*}
u=(u_n)_{n\geq0}&\in\varprojlim_nH^1(\cO_{K(\frm_0\ell^n)}[\frac{1}{\frm\ell}],\bbZ/\ell^n\bbZ(1))\simeq U^\infty_{\{v\mid\frm\ell\}}= H^1(\Delta^\infty)\\
\mbox{and}&\\ 
s=(s_n)_{n\geq0}\in&\varprojlim\bbZ/\ell^n\bbZ[G_{\frm_0\ell^n}]\cdot\sigma=Y^\infty_{\{v\mid\infty\}}\\
\end{align*}
\label{twisting}\end{lemma}
 \begin{proof}  (As in \cite[Lemma 5.1.3]{F2})  The automorphism $\kappa^j$ can be viewed as the inverse limit of similarly defined automorphisms
$\kappa^j$ of the rings $\Lambda_n:=\bbZ/\ell^n\bbZ[G_{\frm_0\ell^n}]$.  Let $f_n:\Spec(\cO_{K(\frm_0\ell^n)}[\frac{1}{\frm\ell}])\to\Spec(\cO_{K(\frm)})$ be the natural map.  The sheaf $\cF_n:=f_{n,*}f_n^*\bbZ/\ell^n\bbZ$ is free of rank one over $\Lambda_n$ with $\pi_1(\Spec(\cO_{K(\frm)}))$-action given by the natural projection $G_\bbQ \to G_{\frm_0\ell^n}$, twisted by the automorphism $g\mapsto g^{-1}$.  There is a $\Lambda_n$-$\kappa^{-j}$-semilinear isomorphism $\mathrm{tw}^j:\cF_n\to\cF_n(j)$ so that Shapiro's lemma gives a commutative diagram of isomorphisms
\begin{equation}\begin{CD} R\Gamma_c(\cO_{K(\frm)}, \cF_n) @>\mathrm{tw}^j>>  R\Gamma_c(\cO_{K(\frm)}, \cF_n(j))\\
@VVV @VVV\\
R\Gamma_c(\cO_{K(\frm_0\ell^n)}[\frac{1}{\frm \ell}], \bbZ/\ell^n\bbZ) @>{\cup \zeta_{\ell^n}^{\otimes j}}>>
R\Gamma_c(\cO_{K_{\frm_0\ell^n}}[\frac{1}{\frm \ell}] ,\bbZ/\ell^n\bbZ(j)),
\end{CD}\label{cd1}\end{equation}
with the horizontal arrows $\Lambda_n$-$\kappa^{-j}$-semilinear. Taking the $\bbZ/\ell^n\bbZ$-dual
of the lower row (with contragredient $G_{\frm_0\ell^n}$-action), we obtain a $\#\circ\kappa^{-j}\circ\#=
\kappa^j$-semilinear isomorphism
$$
R\Gamma(\cO_{K(\frm_0\ell^n)}[\frac{1}{\frm\ell}] ,\bbZ/\ell^n\bbZ(j))\rightarrow
R\Gamma(\cO_{K(\frm_0\ell^n)}[\frac{1}{\frm\ell}] ,\bbZ/\ell^n\bbZ).
$$
After passage to the limit this gives a $\kappa^j$-semilinear isomorphism $\Delta^\infty\simeq\Delta^\infty_j$,
i.e. a $\Lambda$-linear isomorphism $\Delta^\infty\otimes_{\Lambda,\kappa^j}\Lambda
 \simeq\Delta^\infty_j$.  The part a) follows by tensoring over $\Lambda$ with $\bbZ_\ell[G_\frm]$.  For b), consider the inverse map of the lower row of \eqref{cd1} on the degree two cohomology given by
 $$
 H^2_c(\cO_{K(\frm_0\ell^n)}[\frac{1}{\frm\ell}],\bbZ/\ell^n\bbZ)\stackrel{\cup\zeta_{\ell^n}^{\otimes j}}{\leftarrow}H^2_c(\cO_{K(\frm_0\ell^n)}[\frac{1}{\frm\ell}],\bbZ/\ell^n\bbZ(j)).
 $$
 Artin-Verdier duality says that 
 $$
 H^i_c(\cO_{K(\frm_0\ell^n)}[\frac{1}{\frm\ell}],\bbZ/\ell^n\bbZ(j))^\vee=H^{3-i}(\cO_{K(\frm_0\ell^n)}[\frac{1}{\frm\ell}],\bbZ/\ell^n\bbZ(1-j)).
 $$
 Thus we have a dual map which is a $\kappa^j$ semi-linear isomorphism.
  $$
 H^1(\cO_{K(\frm_0\ell^n)}[\frac{1}{\frm\ell}],\bbZ/\ell^n\bbZ(1))\stackrel{\cup\zeta_{\ell^n}^{\otimes j}}{\rightarrow}H^1(\cO_{K(\frm_0\ell^n)}[\frac{1}{\frm\ell}],\bbZ/\ell^n\bbZ(1-j)).
 $$
Moreover, we have a similar diagram to \eqref{cd1} on the level of sheaves where the lower row is obtained by taking invariants under complex conjugation
\begin{equation}\begin{CD} \cF_n @>\mathrm{tw}^j>> \cF_n(j)\\
@VVV @VVV\\
H^0(K(\frm_0\ell^n)\otimes\bbR,\bbZ/\ell^n\bbZ) @>{\cup \zeta_{\ell^n}^{\otimes j}}>>
H^0(K(\frm_0\ell^n)\otimes\bbR,\bbZ/\ell^n\bbZ(j)).
\end{CD}\label{cd2}\end{equation}
Again using the inverse map and taking the $\bbZ/\ell^n\bbZ$ dual, we again have a $\kappa^j$-semilinear isomorphism given by the cup product with $\zeta_{\ell^n}^{\otimes -j}$
$$
\Lambda_n\cdot\sigma\mapsto\Lambda_n\cdot\sigma\cup\zeta_{\ell^n}^{\otimes -j}.
$$
Taking inverse limits, we have part b). 
\end{proof}
We follow with the proof of Theorem \ref{descent}. 
\begin{proof}
As $\Delta(K(\frm))_j$ is a rank 1 $\bbZ_\ell[G_\frm]$-module, the image of $\mathcal{L}\otimes1$ is a basis of the lattice.  It suffices to compare this image with ${_A}\vartheta_\ell\circ{_A}\vartheta_\infty^\#(L^*({_A}M,0)^{-1})$ inside of the rational space $\Delta(K(\frm))\otimes_{\bbZ_\ell}\bbQ_\ell$ which is a rank one module over $A_\ell$ and thus splits over the $\bbQ_\ell$-rational characters $\chi$ of $G_\frm$.  Thus, it suffices to show that
$$
({_A}\vartheta_\ell\circ{_A}\vartheta_\infty(L^*({_A}M,0))_\chi=(\mathcal{L}_{\frm,j})_\chi.
$$
for every $\bbQ_\ell$-rational character $\chi$ of $G_\frm$, where $(\mathcal{L}_{\frm,j})_\chi$ is the image of $\mathcal{L}$ in $\Delta(K(\frm))_j\otimes\bbQ_\ell(\chi)$.  Let $\frq=\frq_{\chi,j}$ be the height 2 prime of $\Lambda$ given by the kernel of the composite ring homomorphism
$$
\chi\kappa^j:\Lambda\stackrel{\kappa^j}{\rightarrow}\Lambda\rightarrow\bbZ_\ell[G(\frm)]\subseteq A_\ell\rightarrow\bbQ_\ell(\chi).
$$
$R:=\Lambda_\frq$ is a regular local ring of dimension 2 with residue field $k:=\bbQ_\ell(\chi)$.  Let $\Delta$ be the module $\Delta^\infty_\frq$ over the localized ring $R$.  To indicate the $\ell$-divisibility of $\frm$ and $\frf_\chi$, we continue with the compound notation above 
$$
\frm=\frm_0\ell^\mu\mbox{   and    }\frf_\chi=\frf_{\chi,0}\ell^{\mu^\prime},
$$
where $(\frm_0,\ell)=(\frf_{\chi,0},\ell)=1$.  For $\ell=\frl_1\frl_2$ split, $\ell^\mu=\frl_1^{\mu_1}\frl_2^{\mu_2}$, and for $\ell=\frl_1^2$ ramified, $\ell^\mu=\frl_1^{\mu_1}$  where $\mu_1$ and $\mu_2$ are integers.
By the Iwasawa Main Conjecture, we can consider $\mathcal{L}$ to be a basis of the $R$-module $(\Det_\Lambda\Delta^\infty)_\frq$ which is isomorphic to $\Det_R\Delta$ since localization is exact.  Lemma \ref{twisting} gives the following isomorphism of complexes of $R$-modules,
$$
\Delta\otimes_R^{\mathbb L}k\stackrel{\simeq}{\rightarrow}\Delta(K(\frm))_j\otimes_{\bbZ_\ell[G_\frm]}k.
$$
\begin{lemma}\label{cohocommute}
$$
H^i(\Delta\otimes^{\mathbb{L}}_Rk)\simeq H^i(\Delta)\otimes_R k.
$$
\end{lemma}
 \begin{proof}
Indeed, if $(x,y)$ is a regular sequence for $R$, then the Koszul complex is the resolution 
$$
0\rightarrow R\stackrel{(\stackrel{x}{y})}{\rightarrow}R\oplus R\stackrel{(y,-x)}{\rightarrow}R\rightarrow k\rightarrow0.
$$
Thus, the homological spectral sequence for Tor degenerates to give an isomorphism 
$$H^2(\Delta\otimes_R^{\mathbb{L}}k)\simeq H^2(\Delta)\otimes k$$
 and in degree 1 an exact sequence
$$
0\rightarrow\Tor_2(H^2(\Delta),k)\rightarrow H^1(\Delta)\otimes k\rightarrow  H^1(\Delta\otimes^{\mathbb{L}}_Rk)\rightarrow\Tor_1(H^2(\Delta),k)\rightarrow0.
$$
Now, the second degree cohomology is given by an exact sequence where the quotient is a free module (lemma \ref{cohomologycomp})
$$
0\rightarrow P^\infty_{\{v\mid\frm\ell\}}\rightarrow H^2(\Delta^\infty)\rightarrow X^\infty_{\{v\mid\frm\ell\infty\}}\rightarrow0.
$$
Again, localization is exact, so we must show that the higher torsion groups of the localized class groups are zero.  As $R$ is a 2-dimensional local ring, the localization $R_\pi$ at a height 1 prime is a DVR, and the image of ${_\fra}\eta_{\frf_{\chi,0}}$ in $H^1(\Delta)_\pi$ is non-zero because of its relationship to the non-vanishing $L$-value.  Then, by \cite[Section 5.5]{JLK} , the fitting ideal of $(P^\infty_\frq)_\pi$ vanishes, and so by Nakayama's lemma does $P^\infty_\frq$.   
\end{proof}
By lemma \ref{cohocommute} the isomorphism of determinants
\begin{align*}
\phi:\Det_k(\Delta\otimes_R^{\mathbb L}k)&\rightarrow\Det_k(\Delta(K(\frm))_j\otimes_{\bbZ_\ell[G_\frm]}k)
\end{align*}
can be computed as a map on the cohomology groups
\begin{align*}
\phi:\bigotimes_{i=1}^2H^i(\Delta)\otimes k&\rightarrow\bigotimes_{i=1}^2H^i(\Delta\otimes^{\mathbb{L}}_R k)\\
&\rightarrow\bigotimes_{i=1}^2H^i(\Delta(K(\frm))_j)\otimes_{\bbQ_\ell[G]} k.
\end{align*}

To compute $\phi(\mathcal{L}\otimes1)$, we consider the elements ${_\fra}\eta_{\frm_0}$ and $\sigma$ independently.  Recall that for an ideal $\frd\mid\frm_0$
$$
N_\frd:=\sum_{\tau\in\Gal(K(\frm_0)/K(\frd))}\tau.
$$
When $\frf_{\chi,0}\mid\frd$, $N_\frd$ is invertible is the ring $R$ since $\chi(N_\frd)=[K(\frm_0):K(\frd)]$.  Thus, in the localized module $\Delta$, the norm compatibility properties of the elliptic units give the equality
\begin{align}
{_\fra}\eta_{\frm_0}=&N_{\frf_{\chi,0}}^{-1}N_{\frf_{\chi,0}}{_\fra}\eta_{\frm_0}\label{localeta}\\
=&N_{\frf_{\chi,0}}^{-1}\prod_{\frp\mid\frm_0,\frp\nmid\frf_{\chi,0}}(1-\Frob_\frp^{-1})(w_{\frm_0}/w_{\frf_{\chi,0}}){_\fra}\eta_{\frf_{\chi,0}}\nonumber\\
=&\begin{multlined}[t](w_{\frm_0}/w_{\frf_{\chi,0}})\left(\sum_{\tau\in\Gal(K(\frm)/K(\frm_0\ell^{\mu^\prime}))}\tau\right)\left(\sum_{\tau\in\Gal(K(\frm)/K(\frm_0\ell^{\mu^\prime}))}\tau\right)^{-1}\nonumber\\
\\N_{\frf_{\chi,0}}^{-1}\prod_{\frp\mid\frm_0,\frp\nmid\frf_{\chi,0}}(1-\Frob_\frp^{-1}){_\fra}\eta_{\frf_{\chi,0}}\end{multlined}\nonumber\\
=&\big(\frac{w_{\frm_0\ell^{\mu^\prime}}}{w_{\frf_\chi}[K(\frm):K(\frf_\chi)]}\big)\Tr_{K(\frm)/K(\frm_0\ell^{\mu^\prime})}\prod_{\frp\mid\frm_0,\frp\nmid\frf_{\chi,0}}(1-\Frob_\frp^{-1}){_\fra}\eta_{\frf_{\chi,0}}.\nonumber
\end{align}
The last equality in \eqref{localeta} can be deduced from the diagram of fields below. 
$$
\xymatrix{
 & K(\frm_0\ell^{\mu^\prime})\ar@{-}[d]^{\frac{w_{\frm_0\ell^{\mu^\prime}}w_{\frf_{\chi,0}}}{w_{\frm_0}w_{\frf_\chi}}}\\
 & K(\frm_0)K(\frf_\chi) \ar@{-}[dl]\ar@{-}[dr]\\
K(\frm_0)\ar@{-}[dr]&&K(\frf_{\chi})\ar@{-}[dl]\\
&K(\frf_{\chi,0})
}
$$
Thus, by Lemma \ref{twisting} 
\begin{align*}
\phi({_\fra}\eta_{\frm_0})=&(w_{\frm_0\ell^{\mu^\prime}}/w_{\frf_\chi})[K(\frm):K(\frf_\chi)]^{-1}\prod_{\frp\mid\frm_0,\frp\nmid\frf_{\chi,0}}(1-\chi(\frp)\cN\frp^{-j})\\
&\cdot\Tr_{K(\frm)/K(\frm_0\ell^{\mu^\prime})}(\Tr_{K(\frm_0\ell^n)/K(\frm)}{_\fra}z_{\frf_{\chi,0}\ell^n}\otimes\zeta_{\ell^n}^{\otimes-j})_n\\
=&[K(\frm):K(\frf_\chi)]^{-1}\prod_{\frp\mid\frm_0,\frp\nmid\frf_{\chi,0}}(1-\chi(\frp)\cN\frp^{-j})\\
&(\Tr_{K(\frf_{\chi,0}\ell^n)/K(\frf_\chi)}{_\fra}z_{\frf_{\chi,0}\ell^n}\otimes\zeta_{\ell^n}^{\otimes-j})_n.
\end{align*}
The second equality follows from a similar diagram of fields
$$
\xymatrix{
 & K(\frm_0\ell^n)\ar@{-}[d]^{\frac{w_{\frf_\chi}}{w_{\frm_0\ell^{\mu^\prime}}}}\\
 & K(\frm_0\ell^{\mu^\prime})K(\frf_{\chi,0}\ell^n) \ar@{-}[dl]\ar@{-}[dr]\\
K(\frm_0\ell^{\mu^\prime})\ar@{-}[dr]&&K(\frf_{\chi,0}\ell^n)\ar@{-}[dl]\\
&K(\frf_{\chi})
}
$$
where we recall that we take $\frm$ and $n$ to be large enough that $w_{\frm_0}=1$ and $w_{\frf_{\chi,0}\ell^n}=1$.  For the second degree cohomology, the situation is somewhat more simple.  Indeed, by lemma \ref{twisting},
\begin{align*}
\phi(\sigma)=&e_\chi(\sigma_\frm\otimes\zeta_{\ell^n}^{\otimes-j})_n\\
=&e_\chi\sigma_\frm\otimes\zeta_{\ell^\infty}^{\otimes-j}.
\end{align*}
Recalling that $\sigma_\frm$ was our fixed choice of embedding $\tau_0$ and multiplying by $\cN\fra-\sigma(\fra)$, we see that in fact
\begin{multline}
\phi(\mathcal{L})=[K(\frm):K(\frf_\chi)]\prod_{\frp\mid\frm_0,\frp\nmid\frf_{\chi,0}}(1-\chi(\frp)\cN\frp^{-j})^{-1}\\(\Tr_{K(\frf_{\chi,0}\ell^n)/K(\frf_\chi)}z_{\frf_{\chi,0}\ell^n}\otimes\zeta_{\ell^n}^{\otimes-j})^{-1}_n\otimes \zeta_{\ell^\infty}^{\otimes-j}\cdot e_\chi\tau_0.
\end{multline}
Since $\chi(\frp)=0$ for $\frp\mid\frf_\chi$ and 2 is a unit $\Lambda_{\frq}$ we have proved that $\phi(\mathcal{L}\otimes1)=({_A}\vartheta_\ell\circ{_A}\vartheta_\infty(L^*({_A}M,0)^{-1}))^\#$.  By its relation to the $L$-value established in theorem \ref{lvaluebasis}, the image of $(\cN\fra-\sigma(\fra))^{-1}{_\fra}\eta_{\frm_0}$ does not vanish and thus is a basis of $H^1(\Delta(K(\frm))_j)\otimes\bbQ_\ell(\chi)$.  Further, the image of $\sigma$ is a basis of $H^2(\Delta(K(\frm))_j)\otimes\bbQ_\ell(\chi)$, completing the proof of the theorem. \end{proof}

\bibliographystyle{amsalpha}

\end{document}